\documentclass[reqno, 12pt]{amsart}

\usepackage[utf8]{inputenc}
\usepackage{lmodern}
\usepackage[english]{babel}
\usepackage{csquotes}
\usepackage{amstext}
\usepackage{amsthm}
\usepackage{amsmath}
\usepackage{amssymb}
\usepackage{graphicx}
\usepackage{mathtools}
\usepackage{float}
\usepackage[svgnames]{xcolor}
\usepackage[margin=1in]{geometry}
\usepackage{dsfont}
\usepackage{fancyvrb}
\usepackage[unicode=true,
    pdfusetitle,
    bookmarks=true,
    bookmarksnumbered=false,
    bookmarksopen=false,
    breaklinks=false,
    pdfborder={0 0 1},
    backref=false,
    colorlinks=true]{hyperref}
\usepackage{url}
\usepackage{enumitem}
\usepackage{tikz}
\usepackage{placeins}
\usepackage{ulem}

\definecolor{color1}{RGB}{27,158,119}
\definecolor{color2}{RGB}{217,95,2}
\definecolor{color3}{RGB}{117,112,179}
\definecolor{color4}{RGB}{231,41,138}

\hypersetup{
    pdftitle={},
    pdfauthor={},
    pdfsubject={},
    pdfkeywords={},
    colorlinks=true,
    linkcolor=DarkBlue,
    citecolor=DarkRed,
    filecolor=DarkMagenta,
    urlcolor=DarkGreen}

\newtheorem{theorem}{Theorem}[section]
\newtheorem{lemma}[theorem]{Lemma}
\newtheorem{proposition}[theorem]{Proposition}

\newtheorem{remark}[theorem]{Remark}
\theoremstyle{definition}

\theoremstyle{definition}

\def\R{\mathbb R}
\def\C{\mathbb C}
\def\N{\mathbb N}
\def\Z{\mathbb Z}

\def\H{\operatorname{H}}
\def\eps{\varepsilon}

\title[Dispersion on the line with Delta potentials]{Dispersion for the Schrödinger equation on the line with short-range array of Delta potentials}

\author[R.~Duboscq]{Romain Duboscq}
\author[\'E.~Durand-Simonnet]{Élio Durand-Simonnet}
\author[S.~Le Coz]{Stefan Le Coz}

\thanks{This work was supported by the ANR LabEx CIMI (grant ANR-11-LABX-0040) within the French State Programme "Investissements d'Avenir" and the ANR project NQG ANR-23-CE40-0005.}

\address[Romain Duboscq]{Institut de Math\'ematiques de Toulouse; UMR5219,
  \newline\indent
  Universit\'e de Toulouse; CNRS,
  \newline\indent
  INSA, F-31077 Toulouse,
  \newline\indent
  France}
\email[Romain Duboscq]{romain.duboscq@math.univ-toulouse.fr}

\address[Élio Durand-Simonnet and Stefan Le Coz]{Institut de Math\'ematiques de Toulouse; UMR5219,
  \newline\indent
  Universit\'e de Toulouse; CNRS,
  \newline\indent
  UPS IMT, F-31062 Toulouse Cedex 9,
  \newline\indent
  France}
  \email[Élio Durand-Simonnet]{elio.durand\_simonnet@math.univ-toulouse.fr}
\email[Stefan Le Coz]{slecoz@math.univ-toulouse.fr}



\begin{document}

\begin{abstract}
    We study dispersive properties of the one-dimensional Schrödinger equation with a short-range array of delta interactions. More precisely, we consider the self-adjoint operator obtained by perturbing the free Laplacian on the line with a real-valued sequence of Dirac delta potentials belonging to weighted $\ell^1 (\Z)$ spaces. Under suitable decay assumptions on the coupling constants and in the absence of a zero-energy resonance, we establish the $L^1 (\R) \to L^\infty (\R)$ dispersive estimate with decay rate $|t|^{-1/2}$ for the associated Schrödinger group. The proof relies on a limiting absorption principle in weighted spaces, explicit representation of the resolvent kernel in terms of Jost solutions and Born series expansion of the Friedrichs extension of the perturbed operator.
\end{abstract}

\maketitle

\tableofcontents

\section{Introduction} \label{sec1}

In this work, we investigate the dispersive properties of the Schrödinger operator with an infinite number of point interactions (or point defects). To be more specific, we are interested in the operator formally given by
\begin{equation*}
    \H_{\alpha} = -\partial_{xx} + \sum_{j\in\mathbb{Z}}\alpha_j \delta(x-j),
\end{equation*}
where $(\alpha_j)_{j\in\mathbb{Z}}$ is a real-valued sequence. This operator corresponds to a Laplace operator perturbed by an infinite sum of singular potentials. This class of operators was first introduced in \cite{KrPe31} (in the particular case $\alpha_j = \alpha$) as a model describing the interaction between an electron and a fixed crystal lattice. Since then, it has been the subject of many mathematical investigations. A systematic study in the case of finitely many point interactions is given in \cite{AlGeHoHo12}, where the authors analyze in detail the spectral properties of the operator. Another comprehensive account of similar properties for related operators can be found in \cite{KoMa13}.

Once our operator is defined properly as a self-adjoint operator (see Section \ref{sec2}), Stone's theorem yields the existence of a propagator $(e^{i t \H_{\alpha}})_{t \in \mathbb{R}}$ which is associated to the time-dependent Schrödinger equation
\begin{equation} \label{eqSch}
    \begin{cases}
        i \partial_t u = \H_{\alpha} u, \\
        u(0, \cdot) = f,
    \end{cases}
\end{equation}
by the relation $u (t, \cdot) = e^{i t \H_{\alpha}} f$, for any $t \in \mathbb{R}$.
It is known that, for the free Schrödinger equation (\textit{i.e.} for $\alpha = 0$), this equation is dispersive which is a property corresponding to the spreading of $u$ over time. It is linked with the so-called dispersive estimates (see \cite{Sc07} for a survey). These estimates are essential tools for the analysis of nonlinear Schrödinger equations (Cauchy problem, scattering theory, see \cite{GiVe79-a, GiVe79-b, Ca03}).

In the context of the operator $\H_{\alpha}$, to the best of our knowledge, the available results around dispersive estimates are restricted to the setting of a finite number of point interactions: \cite{AdSa05, DaHo09} treat the case of a single point interaction, \cite{KoSa10, FePa14} consider two point interactions, and \cite{DuMaWe11, BaIg14} addresses the case of several point interactions. In all these works, with the exception of \cite{DuMaWe11}, the authors rely on explicit (or almost explicit) formulas for the resolvent associated with the operator, which in turn yield an explicit (or nearly explicit) representation of the propagator. For a finite number of point interactions, such explicit formulas can be found, for instance, in \cite{AlGeHoHo12}. However, it is unclear whether these representations can be extended to the case of infinitely many point interactions in a way that would still provide a usable expression for deriving dispersive estimates. In \cite{DuMaWe11}, the authors approach the problem by using the analysis of the so-called wave operators associated to the Schrödinger operators and, to do so, they study the Jost solutions thanks to argument borrowed from \cite{DeTr79}. We finally mention \cite{AnPiTe06}, where the weighted dispersive estimate has been established for the Schrödinger operator on $\R^3$ with a finite number of point defect, together with \cite{IaSc17, DeMiScYa18, CoMiYa19} where unweighted $L^p - L^{p'}$ estimates for restricted regimes of $p$ are obtained in $\R^3$.

Our purpose in this paper is to derive a dispersive estimate for the operator $\H_{\alpha}$ with an infinite number of point interactions. Our main result is stated below.

\begin{theorem} \label{thMain}
    Let $\operatorname{P}$ is the projection onto the essential spectral subspace associated to $\H_{\alpha}$. Assume that there exists $\mu \in (0, 1)$ such that $(j^{1 + \mu} \alpha_j) \in \ell^1 (\Z)$ and that there is no resonance at zero energy, or that $(j^2 \alpha_j) \in \ell^1 (\Z)$. Then, for any $t \in \R^*$ and $f \in L^1 (\R)$, the dispersive estimate
    \begin{equation} \label{eqMain}
        \left\| e^{it \H_\alpha} \operatorname{P} f \right\|_{L^\infty (\R)}
        \lesssim |t|^{-1/2} \| f \|_{L^1 (\R)}
    \end{equation}
    holds.
\end{theorem}

In the previous result, we observe that the presence of the projection $\operatorname{P}$ is necessary as the operator $\H_{\alpha}$ might possess bound states.

Our approach is based on a high/low energy decomposition of the dispersive estimate, following the strategy developed in \cite{GoSc04}. In comparison with \cite[Theorem 1]{GoSc04}, the assumption $V(1+|\cdot|)\in L^1(\mathbb{R})$ naturally corresponds, in our discrete setting, to the condition $(j^{1+\mu}\alpha_j)_{j\in\mathbb{Z}}\in\ell^1(\mathbb{Z})$, for some $\mu>0$. Likewise, in the absence of the no-resonance-at-zero hypothesis, the requirement $V(1+|\cdot|)^2\in L^1(\mathbb{R})$ translates into $(j^2\alpha_j)_{j\in\mathbb{Z}}\in\ell^1(\mathbb{Z})$.

The high-energy contribution, stated in Proposition \ref{prpHighEn}, is treated through an analysis of the Born series, which represents the resolvent associated with $\H_{\alpha}$ as a perturbation of the resolvent of $\H_{0}$. This requires a careful study of the resolvent and, in particular, the establishment of a limiting absorption principle, which is carried out in Theorem \ref{thLAP}. Moreover, particular care is needed when invoking the resolvent identity between $\H_{\alpha}$ and $\H_{0}$, since the two operators do not share the same domain. To overcome this difficulty, we work instead with the resolvents of their respective Friedrichs extensions, for which the identity can be justified rigorously.

The low-energy contribution, addressed in Proposition \ref{prpLowEn}, relies on a representation of the resolvent kernel in terms of Jost solutions (see Theorem \ref{thResKer}). The construction and detailed analysis of the Jost solutions, including an explicit formula, are provided in Proposition \ref{prpJostSol} and Lemma \ref{lemBound}. The no-resonance-at-zero hypothesis (or the stronger summability conditions $\alpha \in \ell^{1, 2} (\Z)$, see Proposition \ref{prpNonZero}), together with bounds on the Fourier transform of the Jost solutions (see Lemmas \ref{lemFourier} and \ref{lemFBound}) allow us to use the Wiener's lemma, recalled in Lemma \ref{lemWiener}, to control terms in the representation of the resolvent kernel and finally deriving the low-energy part of the estimate.


The paper is organized as follows. In Section \ref{sec2}, we fix the notation and introduce the functional spaces used throughout the paper. We also define the free Laplacian and the perturbed Laplacian operators under consideration and recall some of their basic properties. In Section \ref{sec3}, we establish a limiting absorption principle for the perturbed Laplacian, that is, we define the resolvent operator on the essential spectrum in weighted Sobolev spaces, relying on results of \cite{Ma18}. In Section \ref{sec4}, we prove the existence of Jost solutions associated with the perturbed operator, derive bounds satisfied by these solutions using methods of \cite{DeTr79}, and compute the kernel of the resolvent operator on the essential spectrum in terms of the Jost solutions. Finally, in Section \ref{sec5}, we prove Theorem \ref{thMain} by decomposing the Stone formula into two regimes, namely the low-energy and high-energy regimes, following the approach of \cite{GoSc04}.

\section{Notations and definitions} \label{sec2}

\subsection{Functional spaces}

In this subsection, we define the functional spaces needed in our study, particularly the weighted Sobolev spaces and the Hardy space.

\subsubsection{Weighted sequences} Let $s \in \R$. We define the weighted sequence space $\ell^{1, s} (\Z)$ as
\begin{equation*}
    \ell^{1, s} (\Z)
    = \left\{ v \in \ell_{loc}^1 (\Z): (j^s v_j) \in \ell^1 (\Z) \right\}
\end{equation*}
and equip it with the norm $\| \cdot \|_{\ell^{1, s} (\Z)}$ given by
\begin{equation*}
    \| v \|_{\ell^{1, s} (\Z)}
    = \sum_{j \in \Z} |j|^{s} |v_j|.
\end{equation*}

\subsubsection{Schwartz space} We recall the definition of the Schwartz space by
\begin{equation*}
    \mathcal S (\R)
    = \left\{ f \in C^\infty (\R): \sup_{x \in \R} \left| x^m \partial_x^n f (x) \right| < \infty \text{ for all } m, n \in \N \right\}
\end{equation*}
and note that the inclusion
\begin{equation*}
    \mathcal{S} (\R) \subset L^p (\R)
\end{equation*}
is dense for any $1 \leq p < \infty$.

\subsubsection{Weighted Lebesgue and Sobolev spaces} Let $s \in \R$. We define $L^{2, s} (\R)$ as the Hilbert space
\begin{equation} \label{eqWeight}
    L^{2, s} (\R)
    = \left\{ f \in L_{loc}^2 (\R): x \mapsto \left( 1 + x^2 \right)^\frac{s}{2} f (x) \in L^2 (\R) \right\},
\end{equation}
equipped with the scalar product $\langle \cdot, \cdot \rangle_{L^{2, s} (\R)}$ given by
\begin{equation*}
    \left\langle f, g \right\rangle_{L^{2, s} (\R)}
    = \int_{-\infty}^\infty \left( 1 + x^2 \right)^s f (x) \Bar{g} (x) \, dx.
\end{equation*}
For $k \in \N^*$, we define the \textit{weighted Sobolev spaces} $H^{k, s} (\R)$ as
\begin{equation*}
    H^{k, s} (\R)
    = \left\{ f \in L_{loc}^2 (\R): x \mapsto \left( 1 + x^2 \right)^\frac{s}{2} \partial_x^{j} f (x) \in L^2 (\R), \, 0 \leq j \leq k \right\},
\end{equation*}
equipped with the norm $\| \cdot \|_{H^{k, s} (\R)}$ given by
\begin{equation*}
    \| f \|_{H^{k, s} (\R)}^2
    = \sum_{j = 0}^k \left\| (1 + x^2)^\frac{s}{2} \partial_x^{j} f \right\|_{L^2 (\R)}^2.
\end{equation*}
Equivalently, $H^{k, s} (\R)$ can be defined as
\begin{equation*}
    H^{k, s} (\R)
    = \left\{ f \in L_{loc}^2 (\R): x \mapsto \left( 1 + x^2 \right)^\frac{s}{2} f (x) \in H^k (\R) \right\}.
\end{equation*}
In particular, the norm $\| \cdot \|_{H^{1, -s} (\R)}$ and $\| (1 + x^2)^{-s/2} \cdot \|_{H^1 (\R)}$ are equivalent. The topological dual of $H^{1, -s} (\R)$ is denoted by $H^{-1, s} (\R)$ and is equipped with the duality product
\begin{equation*}
    \left\langle f, g \right\rangle_{H^{-1, s} (\R), H^{1, -s} (\R)}
    = \left\langle (1 + |x|^2)^\frac{s}{2} f, (1 + |x|^2)^{-\frac{s}{2}} g \right\rangle_{H^{-1} (\R), H^1 (\R)}.
\end{equation*}
The continuous inclusions
\begin{equation} \label{eqContInc}
    H^1 (\R)
    \hookrightarrow H^{1, -|s|} (\R)
    \hookrightarrow H^{-1, |s|} (\R)
    \hookrightarrow H^{-1} (\R)
\end{equation}
as well as the dense and continuous inclusions
\begin{equation*} 
    H^{1, |s|} (\R)
    \hookrightarrow L^{2, |s|} (\R)
    \hookrightarrow L^2 (\R)
    \hookrightarrow L^{2, -|s|} (\R)
    \hookrightarrow H^{-1, -|s|} (\R)
\end{equation*}
hold.

\subsubsection{Fourier transform} Let $f \in L^1 (\R)$. Its Fourier transform, denoted as $\mathcal{F} (f)$, is given by, for any $\xi \in \R$,
\begin{equation*}
    \mathcal{F} (f) (\xi) = \int_{-\infty}^\infty f(x) e^{-2 \pi i x \xi} \,  \, dx.
\end{equation*}
By abuse of notation, we sometime denote by $\xi \mapsto \mathcal{F} (f (\lambda, x)) (\xi)$ the Fourier transform in $\lambda$ of the parametrized function $f (\cdot, x): \lambda \mapsto f (\lambda, x)$. We recall that, for any $x \in \R$,
\begin{equation*}
    f (x)
    = \int_{-\infty}^\infty \mathcal{F} (f) (\xi) e^{2 \pi i x \xi} \, d \xi.
\end{equation*}

We define the complex upper half-plane $\C^+$ by
\begin{equation*}
    \C^+
    = \left\{ z \in \C : \operatorname{Im} (z) > 0 \right\}.
\end{equation*}

\subsubsection{Hardy space} The \textit{Hardy space} $\mathcal{H} (\C^+)$ is defined as
\begin{equation*}
    \mathcal{H} (\C^+)
    = \left\{ f: \C^+ \to \C \text{ holomorphic}: \sup_{y > 0} \int_{-\infty}^\infty \left| f (x + i y) \right|^2 \,  \, dx < \infty \text{ for any } x \in \R \right\}.
\end{equation*}
Functions $f$ in $\mathcal{H} (\C^+)$ admit non-tangential boundary values on $\R$, still denoted by $f$, which belong to $L^2(\R)$. The space $\mathcal{H} (\C^+)$ can be identified (via boundary values) with the closed subspace of $L^2(\R)$ consisting of functions whose Fourier transform is supported in $\R^+$. For more background on Hardy spaces, see \cite{Du70}.

\subsubsection{Borel measures and total variation norm} Let $\mathcal{M}$ denote the space of finite complex Borel measures on $\mathbb{R}$. For $\mu \in \mathcal{M}$, the total variation measure $|\mu|$ is defined by
\begin{equation*}
    \left| \mu \right| (E)
    = \sup \left\{ \sum_{j = 1}^\infty |\mu(E_j)|: E = \bigsqcup_{j = 1}^\infty E_j, \, E_j \text{ Borel} \right\}.
\end{equation*}
We equip $\mathcal{M}$ with the \textit{total variation norm} $\| \cdot \|_{\mathcal{M}}$, given by
\begin{equation*}
    \| \mu \|_{\mathcal{M}}
    = \left| \mu \right|(\mathbb{R}).
\end{equation*}
In particular, if $\mu$ is absolutely continuous with respect to the Lebesgue measure and admits a density $f \in L^1(\mathbb{R})$, then
\begin{equation*}
    \| \mu \|_{\mathcal{M}}
    = \| f \|_{L^1(\R)}.
\end{equation*}
In this case, by a slight abuse of notation, we identify the measure $\mu$ with its density $f$. We sometimes denote by $\| f (\xi) \|_{\mathcal{M}}$ the norm $\| f \|_{\mathcal{M}}$ of the function $\xi \mapsto f (\xi)$.

\subsubsection{Bounded linear operators} Finally, let $\mathcal{H}_1$ and $\mathcal{H}_2$ be Hilbert spaces. We denote by $B (\mathcal{H}_1, \mathcal{H}_2)$ the space of bounded linear operators from $\mathcal{H}_1$ to $\mathcal{H}_2$, that is,
\begin{equation*}
    B (\mathcal{H}_1, \mathcal{H}_2)
    = \left\{ T : \mathcal{H}_1 \to \mathcal{H}_2 \text{ linear} : \| T \|_{B (\mathcal{H}_1, \mathcal{H}_2)} < \infty \right\},
\end{equation*}
where the operator norm is defined by
\begin{equation*}
    \| T \|_{B (\mathcal{H}_1, \mathcal{H}_2)}
    = \sup \left\{ \| T f \|_{\mathcal{H}_2} : \| f \|_{\mathcal{H}_1} = 1 \right\}.
\end{equation*}

\subsection{Schrödinger operators}

In this subsection, we define the Schrödinger operators studied in this work: the free Laplacian $\H_0$ and the Laplacian perturbed with a short-range array of Delta potentials $\H_{\alpha}$.

\subsubsection{The free Laplace operator} We define the so-called \textit{free Laplacian} as the linear self-adjoint operator $\H_0 : \mathcal{D} (\H_0) \subset L^2 (\R) \to L^2 (\R)$ acting as
\begin{equation*}
    \H_0 f
    = - \partial_{xx} f
\end{equation*}
for $f$ in the domain $\mathcal{D} (\H_0) = H^2 (\R)$. The bilinear form $\operatorname{B}_0: \mathcal{D} (\operatorname{B}_0)^2 \subset L^2 (\R) \to \C$ associated with $\H_{0}$ is given by
\begin{equation*}
    \operatorname{B}_0 (f, g)
    = \langle \partial_x f, \partial_x g \rangle_{L^2 (\R)}.
\end{equation*}
for $f, g \in \mathcal{D} (\operatorname{B}_0) = H^1 (\R)$. The essential spectrum of $\H_0$ is given by
\begin{equation*}
    \sigma_{ess} (\H_0)
    = [0, \infty)
\end{equation*}
and its discrete spectrum is given by
\begin{equation*}
    \sigma_{dis} (\H_0)
    = \emptyset.
\end{equation*}

For $z \in \mathbb C \setminus \sigma (\H_0)$, we define the resolvent operator $\operatorname{R}_0 (z) \in B (L^2 (\R), L^2 (\R))$ as
\begin{equation*}
    \operatorname{R}_0 (z)
    = \left( \H_{0} - z \right)^{-1}.
\end{equation*}
Its associated kernel $\operatorname{G}_0 (z): \R^2 \to \C$ is given by
\begin{equation*}
    \operatorname{G}_0 (z)(x, y)
    = \frac{i}{2 \sqrt{z}} e^{i \sqrt{z} |x - y|}.
\end{equation*}
For $s > 1/2$ and $\lambda > 0$, we can define the operator $\operatorname{R}_0 (\lambda^2 \pm i 0) \in B (L^{2, s} (\R), H^{2, -s} (\R))$ as
\begin{equation*}
    \operatorname{R}_0 (\lambda^2 \pm i 0)
    =  \lim_{\varepsilon \to 0^+} \operatorname{R}_0 (\lambda^2 \pm i \varepsilon) \text{ in } B (L^{2, s} (\R), H^{2, -s} (\R)),
\end{equation*}
see \cite[Theorem $4.1$]{Ag75}. The kernel $\operatorname{G}_0 (\lambda^2 \pm i 0): \R^2 \to \C$ of $\operatorname{R}_0 (\lambda^2 \pm i 0)$ is given by
\begin{equation} \label{eqKerFun}
    \operatorname{G}_0 (\lambda^2 \pm i 0)(x, y)
    = \pm \frac{i}{2 \lambda} e^{\pm i \lambda |x - y|}.
\end{equation}

Let $(e^{i t \H_0})_{t \in \R}$ be the unitary group generated by $\H_0$. Then, for any $t \in \R^*$ and $f \in L^1 (\R)$, the dispersive estimate
\begin{equation} \label{eqFLDis}
    \left\| e^{i t \H_0} f \right\|_{L^\infty (\R)}
    \lesssim |t|^{-1/2} \| f \|_{L^1 (\R)}
\end{equation}
holds, see for example \cite[Theorem $3.1$]{Su07}.

\subsubsection{The Laplace operator with an infinite number of point interactions} Let $\alpha \in \ell^1 (\Z)$ be a real-valued sequence. We define the linear self-adjoint operator $\H_{\alpha} : \mathcal{D} (\H_\alpha) \subset L^2 (\R) \to L^2 (\R)$ as the operator acting as
\begin{equation*}
    \H_{\alpha} f (x)
    = -\partial_{xx} f (x), \quad x \in \R \setminus \Z,
\end{equation*}
for $f$ in the domain
\begin{equation*}
    \mathcal{D} (\H_\alpha)
    = \left\{ f \in H^1 (\R) : \text{for all } j \in \mathbb Z, \,
    \begin{cases}
        f |_{(j, j+1)} \in H^2 (j, j+1), \\
        \partial_x f (j+) - \partial_x f (j-) = \alpha_j f (j)
    \end{cases} 
    \right\}.
\end{equation*}
The bilinear form $\operatorname{B}_\alpha: \mathcal{D} (\operatorname{B}_\alpha)^2 \subset L^2 (\R) \to \C $ associated with $\H_{\alpha}$ is given by
\begin{equation} \label{eqBlnFrm}
    \operatorname{B}_\alpha (f, g)
    = \left\langle \partial_x f, \partial_x g \right\rangle_{L^2 (\R)} + \sum_{j \in \mathbb Z} \alpha_j f (j) \overline{g} (j)
\end{equation}
for $f, g \in \mathcal{D} (\operatorname{B}_\alpha) = H^1 (\R)$.

We now determine the spectrum $\sigma(\H_\alpha)$ of $\H_\alpha$. The essential spectrum $\sigma_{ess} (\H_\alpha)$ is given by
\begin{equation*}
    \sigma_{ess} (\H_\alpha)
    = [0, \infty).
\end{equation*}
Indeed, since $\alpha \in \ell^1 (\Z)$, the operator $\H_\alpha$ is relatively compact with respect to $\H_0$ in the sense of Kato-Rellich. By Weyl's theorem on relatively compact perturbations, $\H_\alpha$ and $\H_0$ share the same essential spectrum. Moreover, as $\alpha$ is real-valued and belongs to $\ell^1 (\Z)$, any eigenvalue of $\H_\alpha$ outside $[0, \infty)$ is isolated and of finite multiplicity. The quadratic form
\begin{equation*}
    \operatorname{B}_\alpha (f, f)
    = \left\| \partial_x f \right\|_{L^2 (\R)}^2 + \sum_{j \in \mathbb{Z}} \alpha_j |f (j)|^2
\end{equation*}
is bounded from below by the inclusion $H^1 (\R) \hookrightarrow L^\infty (\R)$. Consequently, the discrete spectrum $\sigma_{dis} (\H_\alpha)$ consists of a countable set of negative eigenvalues and is bounded from below.

For $z \in \mathbb C \setminus \sigma (\H_\alpha)$, we define the resolvent operator $\operatorname{R}_\alpha (z) \in B (L^2 (\R), L^2 (\R))$ as
\begin{equation*}
    \operatorname{R}_\alpha (z)
    = \left( \H_{\alpha} - z \right)^{-1}.
\end{equation*}

\begin{remark}
    Observe that we cannot apply the resolvent identity to compare $\operatorname{R}_0 (z)$ and $\operatorname{R}_\alpha (z)$ as $\H_0$ and $\H_{\alpha}$ are not defined on the same domain.
\end{remark}


\section{Limiting absorption principle} \label{sec3}

This section aims to establish the limiting absorption principle for $\H_{\alpha}$, stated in the following theorem. We also give the asymptotic behavior of $\operatorname{R}_\alpha (\lambda^2 \pm i 0) f$ for $f \in L^{2, s} (\R)$, as it will be needed in order to compute its kernel.

\begin{theorem} \label{thLAP}
    Assume that there exists $\mu > 0$ such that $\alpha \in \ell^{1, 1 + \mu} (\Z)$. Let $\lambda \geq 0$ and $s \in (1/2, (1 + \mu)/2)$. Then, the operators $\operatorname{R}_\alpha (\lambda^2 \pm i 0) \in B (L^{2, s} (\R), L^{2, -s} (\R))$ given by
    \begin{equation} \label{eqLAP1}
        \operatorname{R}_\alpha (\lambda^2 \pm i 0)
        =  \lim_{\varepsilon \to 0^+} \operatorname{R}_\alpha (\lambda^2 \pm i \varepsilon) \text{ in } B (L^{2, s} (\R), L^{2, -s} (\R))
    \end{equation}
    are well-defined. Furthermore, for $f \in L^{2, s} (\R)$, $\operatorname{R}_\alpha (\lambda^2 + i 0) f$ satisfies the outgoing condition
    \begin{align}
        & (\partial_x - i \lambda) \operatorname{R}_\alpha (\lambda^2 + i 0) f (x) \to 0 \text{ as } x \to \infty, \label{eqOutG1} \\
        & (\partial_x + i \lambda) \operatorname{R}_\alpha (\lambda^2 + i 0) f (x) \to 0 \text{ as } x \to -\infty; \label{eqOutG2}
    \end{align}
    and $\operatorname{R}_\alpha (\lambda^2 - i 0) f$ satisfies the incoming condition
    \begin{align}
        & (\partial_x + i \lambda) \operatorname{R}_\alpha (\lambda^2 - i 0) f (x) \to 0 \text{ as } x \to \infty, \label{eqInC1} \\
        & (\partial_x - i \lambda) \operatorname{R}_\alpha (\lambda^2 - i 0) f (x) \to 0 \text{ as } x \to -\infty. \label{eqInC2}
    \end{align}
\end{theorem}

To do so, we first establish the limiting absorption principle for the Friedrichs extension of $\H_{\alpha}$ over $H^{-1} (\R)$, with an extensive use of results from \cite{Ma18}. See \cite[Chapter VI, Section $3$]{Ka66} for properties of Friedrichs extension.

We first define the self-adjoint operator $\tilde{\H}_0 : H^{-1} (\R) \to H^{-1} (\R)$ as the Friedrichs extension of $\H_0$ over $H^{-1} (\R)$ with domain $D(\tilde{\H}_0) = H^1 (\R)$ and acting as
\begin{equation*}
    \left\langle \tilde{\H}_0 f, g \right\rangle_{H^{-1} (\R), H^1 (\R)}
    = \operatorname{B}_0 (f, g)
\end{equation*}
for $f, g \in H^1 (\R)$. We have
\begin{equation*}
    \sigma_{ess} (\tilde{\H}_0)
    = \sigma_{ess} (\H_0)
    = [0, \infty),
    \quad
    \sigma_{dis} (\tilde{\H}_0)
    = \sigma_{dis} (\H_0)
    = \emptyset.
\end{equation*}
For $s > 1/2$ and $z \in \C \setminus [0, \infty)$, the resolvent operator $\tilde{\operatorname{R}}_0 (z) \in B (H^{-1, s} (\R), H^{1, -s} (\R))$ given by
\begin{equation*}
    \tilde{\operatorname{R}}_0 (z)
    = (\tilde{\H}_0 - z)^{-1}
\end{equation*}
is well-defined and, for $w \in \C \setminus [0, \infty)$, the resolvent identity
\begin{equation} \label{eqResIdExtFL}
    \tilde{\operatorname{R}}_0 (z)
    = \tilde{\operatorname{R}}_0 (w) + (z - w) \tilde{\operatorname{R}}_0 (z) \tilde{\operatorname{R}}_0 (w)
\end{equation}
holds. Furthermore, $\tilde{\operatorname{R}}_0 (z) \, \delta \in H^1 (\R)$ satisfies, for $x \in \R$,
\begin{equation} \label{eqKerDis}
    \tilde{\operatorname{R}}_0 (z) \, \delta (x)
    = \frac{i}{2 \sqrt{z}} e^{i \sqrt{z} |x|}.
\end{equation}

Similarly, we define the self-adjoint operator $\tilde{\H}_\alpha : H^{-1} (\R) \to H^{-1} (\R)$ as the Friedrichs extension of $\H_{\alpha}$ over $H^{-1} (\R)$ with domain $D(\tilde{\H}_\alpha) = H^1 (\R)$ and acting as
\begin{equation*}
    \left\langle \tilde{\H}_\alpha f, g \right\rangle_{H^{-1} (\R), H^1 (\R)}
    = \operatorname{B}_\alpha (f, g)
\end{equation*}
for $f, g \in H^1 (\R)$. We have
\begin{equation*}
    \sigma_{ess} (\tilde{\H}_\alpha)
    = \sigma_{ess} (\H_\alpha)
    = [0, \infty),
    \quad
    \sigma_{dis} (\tilde{\H}_\alpha)
    = \sigma_{dis} (\H_\alpha).
\end{equation*}
For $z \in \C \setminus \sigma (\H_\alpha)$, the resolvent operator $\tilde{\operatorname{R}}_\alpha (z) \in B (H^{-1} (\R), H^1 (\R))$ is given by
\begin{equation*}
    \tilde{\operatorname{R}}_\alpha (z)
    = (\tilde{\H}_\alpha - z)^{-1}.
\end{equation*}
By the inclusion \eqref{eqContInc}, we also have that $\tilde{\operatorname{R}}_\alpha (z) \in B (H^{-1, s} (\R), H^{1, -s} (\R))$ for  $s > 1/2$.

Observe that $\tilde{\H}_0$ and $\tilde{\H}_\alpha$ share the domain $H^1 (\R)$. We may therefore regard $\tilde{\H}_\alpha - \tilde{\H}_0$ as an operator defined on this common domain, given by
\begin{equation*}
    (\tilde{\H}_\alpha - \tilde{\H}_0) f
    = \sum_{j \in \Z} \alpha_j f (j) \delta (\cdot - j)
\end{equation*}
for $f \in H^1 (\R)$. We also define the operator $(1 + x^2)^{(1 + \mu)/2} (\tilde{\H}_\alpha - \tilde{\H}_0)$ as
\begin{equation*}
    (1 + x^2)^{(1 + \mu)/2} (\tilde{\H}_\alpha - \tilde{\H}_0) f
    = \sum_{j \in \Z} \left( 1 + |j|^2 \right)^\frac{1 + \mu}{2} \alpha_j f (j) \delta(\cdot - j).
\end{equation*}
for $f \in H^1 (\R)$. We establish that those operators are well-defined in the following lemma, and that $\tilde{\H}_\alpha - \tilde{\H}_0$ is symmetric and compact.

\begin{lemma} \label{lemComp}
    Assume that there exists $\mu > 0$ such that $\alpha \in \ell^{1, 1 + \mu} (\Z)$. The following statements hold.
    \begin{enumerate}
        \item The operator $\tilde{\H}_\alpha - \tilde{\H}_0 \in B (H^{1} (\R), H^{-1} (\R))$ is well-defined, symmetric and compact.
        \item The operator $(1 + x^2)^{(1 + \mu)/2} (\tilde{\H}_\alpha - \tilde{\H}_0) \in B (H^{1} (\R), H^{-1} (\R))$ is well-defined. 
    \end{enumerate}
\end{lemma}

\begin{proof}
    We start by proving the first statement. Let $f, g \in H^{1} (\R)$, we have
    \begin{align*}
        \left|\left\langle (\tilde{\H}_\alpha - \tilde{\H}_0) f, g \right\rangle_{H^{-1} (\R), H^1 (\R)} \right|
        \leq \sum_{j \in \Z} |\alpha_j| |f (j)| |g (j)|
        \leq \| \alpha_j \|_{\ell^1 (\Z)} \| f \|_{H^1 (\R)} \| g \|_{H^1 (\R)}
    \end{align*}
    by the fact that $\alpha \in \ell^1 (\Z)$ and the continuous inclusion $H^1 (\R) \hookrightarrow L^\infty (\R)$. Thus, the operator $\tilde{\H}_\alpha - \tilde{\H}_0 \in B (H^{1} (\R), H^{-1} (\R))$ is well-defined. Furthermore, it is symmetric as $\tilde{\H}_0$ and $\tilde{\H}_\alpha$ are self-adjoint. We now prove that it is compact. Let $S \in B (H^{1} (\R), \ell^{2} (\Z))$ be the sampling operator defined by
    \begin{equation*}
        S (f)  = (f (j))_{j \in \Z}.
    \end{equation*}
    It is well-defined as, for $f \in H^1 (\R)$ and $j \in \Z$, one gets
    \begin{equation} \label{eqSamp}
        |f (j)|^2
        \lesssim \| f \|_{H^1 (j, j+1)}^2.
    \end{equation}
    Let $D_\alpha \in B (\ell^2 (\Z), \ell^2 (\Z))$ be the diagonal operator defined by
    \begin{equation*}
        D_\alpha \left( (v_j) \right) = \left( \alpha_j v_j \right).
    \end{equation*}
    As $\alpha_j \to 0$ as $|j| \to \infty$, the operator $D$ is compact and so is $D \circ S$. Finally, let $A \in B (\ell^2 (\Z), H^{-1} (\R))$ be the operator defined by
    \begin{equation*}
        A (v)
        = \sum_{j \in \Z} v_j \delta(\cdot - j).
    \end{equation*}
    It is well-defined as, for $v \in \ell^2 (\Z)$ and $f \in H^1 (\R)$, one gets
    \begin{equation*}
        \left| \left\langle A (v), f \right\rangle_{H^{-1} (\R), H^1 (\R)} \right|
        \leq \| v \|_{\ell^2 (\Z)} \| f (j) \|_{\ell^2 (\Z)}
        \leq \| v \|_{\ell^2 (\Z)} \| f \|_{H^1 (\R)}^2,
    \end{equation*}
    where we used Cauchy-Schwartz estimate for the second inequality and \eqref{eqSamp} for the third one.
    We have $\tilde{\H}_\alpha - \tilde{\H}_0 = A \circ D_\alpha \circ S$, so $\tilde{\H}_\alpha - \tilde{\H}_0$ is the composition of a bounded operator with a compact operator, therefore it is compact.

    We now prove the second statement. Let $f, g \in H^{1} (\R)$, we have
    \begin{align*}
        \left| \left\langle (1 + x^2)^{(1 + \mu)/2} (\tilde{\H}_\alpha - \tilde{\H}_0) f, g \right\rangle_{H^{-1} (\R), H^1 (\R)} \right|
        & \leq \sum_{j \in \Z} \left( 1 + |j|^2 \right)^\frac{1 + \mu}{2} |\alpha_j| |f (j)| |g (j)| \\
        & \lesssim \left( \sum_{j \in \Z} \left( 1 + |j|^2 \right)^\frac{1 + \mu}{2} |\alpha_j| \right) \| f \|_{H^1 (\R)} \| g \|_{H^1 (\R)},
    \end{align*}
    by the continuous inclusion $H^1 (\R) \hookrightarrow L^\infty (\R)$. Thus, the operator $(1 + x^2)^{(1 + \mu)/2} (\tilde{\H}_\alpha - \tilde{\H}_0) \in B (H^{1} (\R), H^{-1} (\R))$ is well-defined. This concludes the proof.
\end{proof}

We then establish that $\tilde{\H}_\alpha - \tilde{\H}_0$ belongs to the space $B (H^{1, -s} (\R), H^{-1, s} (\R))$ for a suitable parameter $s$ and deduce the associated resolvent identity.

\begin{lemma} \label{lemResExtId}
    Assume that there exists $\mu > 0$ such that $\alpha \in \ell^{1, 1 + \mu} (\Z)$. Let $s \in (0, (1 + \mu)/2)$. Then, the operator $\tilde{\H}_\alpha - \tilde{\H}_0 \in B (H^{1, -s} (\R), H^{-1, s} (\R))$ is well-defined.
\end{lemma}

\begin{proof}
    Let $f, g \in H^{1, -s} (\R)$, we have
    \begin{multline*}
        \left| \left\langle (\tilde{\H}_\alpha - \tilde{\H}_0) f, g \right\rangle_{H^{-1, s} (\R), H^{1, -s} (\R)} \right|
        \leq \sum_{j \in \Z} \left( 1 + |j|^2 \right)^s |\alpha_j| \frac{|f (j)|}{\left( 1 + |j|^2 \right)^{\frac{s}{2}}} \frac{|g (j)|}{\left( 1 + |j|^2 \right)^{\frac{s}{2}}} \\
        \lesssim \left( \sum_{j \in \Z} \left( 1 + |j|^2 \right)^s |\alpha_j| \right) \| (1 + x^2)^{-\frac{s}{2}} f \|_{L^\infty (\R)} \| (1 + x^2 )^{-\frac{s}{2}} g \|_{L^\infty (\R)} \\
        \lesssim \left( \sum_{j \in \Z} \left( 1 + |j|^2 \right)^\frac{1 + \mu}{2} |\alpha_j| \right) \left\| f \right\|_{H^{1, -s} (\R)} \left\| g \right\|_{H^{1, -s} (\R)} \\
        \lesssim \| \alpha \|_{\ell^{1, 1 + \mu} (\Z)} \left\| f \right\|_{H^{1, -s} (\R)} \left\| g \right\|_{H^{1, -s} (\R)}
    \end{multline*}
    by the continuous inclusion $H^1 (\R) \hookrightarrow L^\infty (\R)$ and the fact that the norm $\| \cdot \|_{H^{1, -s} (\R)}$ and $\| (1 + x^2)^{-s/2} \cdot \|_{H^1 (\R)}$ are equivalent. The resolvent identity follows and this concludes the proof.
\end{proof}

\begin{remark}
    In particular, for any $z \in \C \setminus \sigma (\H_\alpha)$, the resolvent identity
    \begin{equation} \label{eqResExtId1}
        \tilde{\operatorname{R}}_\alpha (z)
        = \tilde{\operatorname{R}}_0 (z) - \tilde{\operatorname{R}}_0 (z) (\tilde{\H}_\alpha - \tilde{\H}_0) \tilde{\operatorname{R}}_\alpha (z)
    \end{equation}
    holds in the sense of $B (H^{-1, s} (\R), H^{1, -s} (\R))$.
\end{remark}

We can now establish the limiting absorption principle for $\tilde{\H}_0$ and $\tilde{\H}_\alpha$ and the associated resolvent identity, as stated in the following proposition. The proof is based on results from \cite{Ma18}.

\begin{proposition} \label{prpResExt}
    Assume that there exists $\mu > 0$ such that $\alpha \in \ell^{1, 1 + \mu} (\Z)$. Let $\lambda \geq 0$ and $s > 1/2$. The following statements hold.
    \begin{enumerate}
        \item The operators $\tilde{\operatorname{R}}_0 (\lambda^2 \pm i 0) \in B (H^{-1, s} (\R), H^{1, -s} (\R))$ given by
        \begin{equation} \label{eqResExt2}
            \tilde{\operatorname{R}}_0 (\lambda^2 \pm i 0)
            = \lim_{\varepsilon \to 0^+} \tilde{\operatorname{R}}_0 (\lambda^2 \pm i \varepsilon) \text{ in } B (H^{-1, s} (\R), H^{1, -s} (\R))
        \end{equation}
        are well-defined.
        \item The operators $\tilde{\operatorname{R}}_\alpha (\lambda^2 \pm i 0) \in B (H^{-1, s} (\R), H^{1, -s} (\R))$ given by
        \begin{equation} \label{eqResExt3}
            \tilde{\operatorname{R}}_\alpha (\lambda^2 \pm i 0)
            = \lim_{\varepsilon \to 0^+} \tilde{\operatorname{R}}_\alpha (\lambda^2 \pm i \varepsilon) \text{ in } B (H^{-1, s} (\R), H^{1, -s} (\R))
        \end{equation}
        are well-defined. 
    \end{enumerate}
    Furthermore, for all $s \in (1/2, (1 + \mu)/2)$, the resolvent identities
    \begin{equation} \label{eqResExt1}
        \tilde{\operatorname{R}}_\alpha (\lambda^2 \pm i 0)
        = \tilde{\operatorname{R}}_0 (\lambda^2 \pm i 0) - \tilde{\operatorname{R}}_0 (\lambda^2 \pm i 0) (\tilde{\H}_\alpha - \tilde{\H}_0) \tilde{\operatorname{R}}_0 (\lambda^2 \pm i 0)
    \end{equation}
    holds in $B (H^{-1, s} (\R), H^{1, -s} (\R))$.
\end{proposition}

\begin{proof}
    The operator $\tilde{\H}_0$ satisfies the hypothesis of \cite[Corollary $5.9$]{Ma18} (take $V = 0$, where $V$ denotes the potential in \cite{Ma18}) and, by Lemma \ref{lemComp}, the operator $\H_{\alpha}$ also satisfies the hypothesis of \cite[Corollary $5.9$]{Ma18} (take $V = \tilde{\H}_\alpha - \tilde{\H}_0$). Thus, we can apply \cite[Theorem $1.2$]{Ma18} and \cite[Corollary $3.1$]{Ma18} and we obtain \eqref{eqResExt2} and \eqref{eqResExt3}. Finally, we pass to the limit in \eqref{eqResExtId1} in order to obtain \eqref{eqResExt1}. This concludes the proof.
\end{proof}

We can finally prove Theorem \ref{thLAP}.

\begin{proof}[Proof of Theorem \ref{thLAP}]
    Let $f \in  L^{2, s} (\R)$. Then $\tilde{\operatorname{R}}_\alpha (\lambda^2 \pm i 0) f \in H^{1, -s} (\R) \subset L^{2, -s} (\R)$ are well-defined by Proposition \ref{prpResExt}. We define $\operatorname{R}_\alpha (\lambda^2 \pm i 0) f = \tilde{\operatorname{R}}_\alpha (\lambda^2 \pm i 0) f$, so that
    \begin{multline*}
        \| \operatorname{R}_\alpha (\lambda^2 \pm i \eps) f - \operatorname{R}_\alpha (\lambda^2 \pm i 0) f \|_{L^{2, -s} (\R)} \\
        \leq \| \tilde{\operatorname{R}}_\alpha (\lambda^2 \pm i \eps) f - \tilde{\operatorname{R}}_\alpha (\lambda^2 \pm i 0) f \|_{H^{1, -s} (\R)} \\
        \leq \| \tilde{\operatorname{R}}_\alpha (\lambda^2 \pm i \eps) - \tilde{\operatorname{R}}_\alpha (\lambda^2 \pm i 0) \|_{B (H^{-1, s} (\R), H^{1, -s} (\R))} \| f \|_{H^{-1, s} (\R)} \\
        \leq \| \tilde{\operatorname{R}}_\alpha (\lambda^2 \pm i \eps) - \tilde{\operatorname{R}}_\alpha (\lambda^2 \pm i 0) \|_{B (H^{-1, s} (\R), H^{1, -s} (\R))} \| f \|_{L^{2, s} (\R)}
        \to 0
    \end{multline*}
    as $\eps \to 0$. Therefore, the limit result \eqref{eqLAP1} holds.
    
    We now establish the asymptotic behavior of $\operatorname{R}_\alpha (\lambda^2 + i 0) f$. By the resolvent identity \eqref{eqResExt1}, we have, for $x \in \R$,
    \begin{align*}
        \operatorname{R}_\alpha (\lambda^2 + i 0) f (x)
        & = \tilde{\operatorname{R}}_\alpha (\lambda^2 + i 0) f (x) \\
        & = \tilde{\operatorname{R}}_0 (\lambda^2 + i 0) f (x) - \tilde{\operatorname{R}}_0 (\lambda^2 + i 0) (\tilde{\H}_\alpha - \tilde{\H}_0) \tilde{\operatorname{R}}_\alpha (\lambda^2 + i 0) f (x) \\
        & = \int_{-\infty}^\infty \frac{i e^{i \lambda |x - y|}}{2 \lambda} f (y) \, dy + \sum_{j \in \Z} \frac{i \alpha_j}{2 \lambda} \operatorname{R}_\alpha (\lambda^2 + i 0) f (j) e^{i \lambda |x - j|},
    \end{align*}
    where we used \eqref{eqKerFun} coupled with \eqref{eqKerDis}. Applying the operator $(\partial_x - i \lambda)$ to the expression above yields
    \begin{equation*}
        (\partial_x - i \lambda) \operatorname{R}_\alpha (\lambda^2 + i 0) f (x)
        = \sum_{j > x} \alpha_j \operatorname{R}_\alpha (\lambda^2 + i 0) f (j) e^{i \lambda (j - x)},
    \end{equation*}
    leading to the asymptotic behavior stated in \eqref{eqOutG1}. The remaining cases are handled in the same manner. This concludes the proof.
\end{proof}

\section{Jost solutions} \label{sec4}


Let $\lambda \in \C^+ \cup \R$. We define the \textit{Jost solutions} $\operatorname{f}_+ (\lambda, \cdot)$ and $\operatorname{f}_- (\lambda, \cdot)$ associated to $\H_{\alpha}$ as the unique solutions of
\begin{equation} \label{eqJost1}
    -\partial_{xx} \operatorname{f}_\pm (\lambda, x)
    = \lambda^2 \operatorname{f}_\pm (\lambda, x), \quad x \in \R \setminus \Z
\end{equation}
which satisfy the continuity condition
\begin{equation} \label{eqJost2}
    \operatorname{f}_\pm (\lambda, j-)
    = \operatorname{f}_\pm (\lambda, j+), \quad j \in \Z,
\end{equation}
and the jump derivative condition
\begin{equation} \label{eqJost3}
    \partial_x \operatorname{f}_\pm (\lambda, x) \big|_{x = j+} - \partial_x \operatorname{f}_\pm (\lambda, x) \big|_{x = j-} = \alpha_j \operatorname{f}_\pm (\lambda, j), \quad j \in \Z,
\end{equation}
together with the asymptotic behavior
\begin{align} 
    & \operatorname{f}_+ (\lambda, x) \sim e^{i \lambda x} \text{ as } x \to \infty, \label{eqJost4} \\
    & \operatorname{f}_- (\lambda, x) \sim e^{-i \lambda x} \text{ as } x \to -\infty. \label{eqJost5}
\end{align}

In this section, we first prove the existence of Jost solutions (see Proposition \ref{prpJostSol}) and state some useful properties. In particular, we prove a property on the Wronskian on the Jost solutions which will allow us to get rid of the non-resonance assumption under stronger summability of $\alpha$ (see Proposition \ref{prpNonZero}). Finally, we derive a representation of the kernel of $\operatorname{R} (\lambda^2 \pm i0)$ based on Jost solutions (see Theorem \ref{thResKer}).

\subsection{Existence and properties}

We first prove the existence of the Jost solutions. To do so, we establish the following technical lemma.

\begin{lemma} \label{lemMatrix}
    Assume that $\alpha \in \ell^{1, 1} (\Z)$. Let $\lambda \in \C^+ \cup \R$ and $j \in \Z$. Let $N_j (\lambda) \in \operatorname{M}_2 (\C)$ be the degree $1$-nilpotent matrix given by
    \begin{equation*}
        N_j (\lambda) 
        =
        \begin{cases}
            \dfrac{\operatorname{sign} (j)}{2 i \lambda}
            \begin{pmatrix}
                1 & e^{-2 i \lambda j} \\
                -e^{2 i \lambda j} & -1
            \end{pmatrix} & \text{ if } \lambda \neq 0, \\
            \operatorname{sign} (j)
            \begin{pmatrix}
                -j & -j^2 \\
                1 & j
            \end{pmatrix} & \text{ if } \lambda = 0.
        \end{cases}        
    \end{equation*}
    Let $M_j (\lambda)$ and $P_j (\lambda) \in \operatorname{GL}_2 (\C)$ be given by
    \begin{equation} \label{eqMatrix1}
        M_j (\lambda)
        = I_2 + \alpha_j N_j (\lambda), \quad
        P_j (\lambda)
        =
        \begin{cases}
            M_j (\lambda) M_{j-1} (\lambda) \dots M_2 (\lambda) M_1 (\lambda) & \text{ if } j \geq 1, \\
            M_j (\lambda) M_{j+1} (\lambda) \dots M_{-2} (\lambda) M_{-1} (\lambda) & \text{ if } j \leq -1.
        \end{cases}
    \end{equation}
    Then, the following assertions hold.
    \begin{enumerate}
        \item If $\lambda \in \R^*$, there exists $P_{\pm \infty} (\lambda) \in \operatorname{GL}_2 (\C)$ such that
        \begin{align*}
            & P_j (\lambda) \to P_\infty(\lambda) \text{ as } j \to \infty, \\
            & P_j (\lambda) \to P_{-\infty} (\lambda) \text{ as } j \to - \infty.
        \end{align*}
        Furthermore, the functions $\lambda \mapsto P_\infty (\lambda)^{-1}$ and $\lambda \mapsto P_{-\infty} (\lambda)^{-1}$ are $C^1$ on $\R^*$.
        \item If $\lambda \in \C^+ \cup \{ 0 \}$, there exists vectors $P_\infty (\lambda) (1, 0)^T, P_{-\infty} (\lambda) (0, 1)^T \in \C^2$ such that
        \begin{align*}
            & P_j (\lambda)^{-1} (1, 0)^T \to P_\infty^{-1} (\lambda) (1, 0)^T \text{ as } j \to \infty, \\
            & P_j (\lambda)^{-1} (0, 1)^T \to P_{-\infty}^{-1} (\lambda) (0, 1)^T \text{ as } j \to - \infty.
        \end{align*}
        Furthermore, the functions $\lambda \mapsto P_\infty (\lambda)^{-1} (1, 0)^T$ and $\lambda \mapsto P_{-\infty} (\lambda)^{-1} (0, 1)^T$ are holomorphic on the upper half plane $\C^+$.
    \end{enumerate}
\end{lemma}

In what follows, we shall write $N_j (\lambda)$ (respectively $M_j (\lambda)$ and $P_j (\lambda)$) simply as $N_j$ (respectively $M_j$ and $P_j$) for $j \in \Z \cup \{ \pm \infty \}$ whenever the dependence on $\lambda$ is clear.

\begin{proof}[Proof of Lemma \ref{lemMatrix}]
    Let $j \geq 2$ and $\lambda \in \C^+ \cup \R$. Observe that
    \begin{equation*}
        M_j^{-1}
        = I_2 - \alpha_j N_j,
        \quad
        P_j^{-1}
        = M_1^{-1} M_2^{-1} \dots M_{j-1}^{-1} M_j^{-1}.
    \end{equation*}
    Thus,
    \begin{align*}
        P_j^{-1}
        = P_{j-1}^{-1} M_j^{-1}
        = P_{j-1}^{-1} (I_2 - \alpha_j N_j)
    \end{align*}
    so
    \begin{align*}
        P_j^{-1} - P_{j-1}^{-1}
        & = -\alpha_j P_{j-1}^{-1} N_j \\
        & = -\alpha_j (I_2 - \alpha_1 N_1) \cdots (I_2 - \alpha_{j-1} N_{j-1}) N_j \\
        & = -\alpha_j  \left( \sum_{1 \leq n \leq j-1} (-1)^n \sum_{1 \leq j_1 < \ldots < j_n \leq j-1} \alpha_{j_1} \cdots \alpha_{j_n} N_{j_1} \cdots N_{j_n} \right) N_j.
    \end{align*}
    
    We divide the rest of the proof in two steps: the first one shows Assertion $(1)$ while the second proves Assertion $(2)$. The matrix norm used is the spectral norm.
    
    \medskip
    
    \noindent \emph{Proof of $(1)$.} Assume that $\lambda \in \R^*$. By a direct calculation, we have
    \begin{equation*}
        \| N_j (\lambda) \|
        = \frac{1}{|\lambda|}.
    \end{equation*}
    We obtain
    \begin{align*}
        \left\| P_j (\lambda)^{-1} - P_{j-1} (\lambda)^{-1} \right\|
        & \leq \left| \frac{\alpha_j}{\lambda} \right| \sum_{1 \leq n \leq j-1} \, \sum_{1 \leq j_1 < \ldots < j_n \leq j-1} \left| \frac{\alpha_{j_1}}{\lambda} \right| \cdots \left| \frac{\alpha_{j_n}}{\lambda} \right| \\
        & \leq \left| \frac{\alpha_j}{\lambda} \right| \sum_{1 \leq n \leq j-1} \frac{1}{n!} \left( \frac{\| \alpha \|_{\ell^1 (\Z)}}{|\lambda|} \right)^n \\
        & \leq \left| \frac{\alpha_j}{\lambda} \right| \exp \left( \frac{\| \alpha \|_{\ell^1 (\Z)}}{|\lambda|} \right)
    \end{align*}
    as $\alpha \in \ell^1 (\Z)$. Thus,
    \begin{equation}  \label{eqMatrix2}
        \sum_{j \geq 2} \left\| P_{j}^{-1} (\lambda) - P_{j-1}^{-1} (\lambda) \right\|
        \leq \exp \left( \frac{\| \alpha \|_{\ell^1 (\Z)}}{2|\lambda|} \right) \sum_{j \geq 2} \left| \frac{\alpha_j}{\lambda} \right|
        < \infty.
    \end{equation}
    
    By \eqref{eqMatrix2}
    , for any $\lambda \in \R^*$, we obtain that $(P_j (\lambda)^{-1})_{j \geq 1}$ is a Cauchy sequence which therefore converges. We may prove similarly that $P_{-\infty} (\lambda)^{-1}$ is well-defined. The $C^1$ nature follows from the fact that $\lambda \mapsto N_j (\lambda)$ is $C^1$ on the upper half plane. Thus, Assertion $(1)$ holds.

    \medskip
    
    \noindent \emph{Proof of $(2)$.}
        First, assume that $\lambda \in \C^+$. We have
    \begin{equation*}
        N_j (\lambda)
        = \frac{1}{2 i \lambda}
        \begin{pmatrix}
            1 \\
            -e^{i 2 \lambda j}
        \end{pmatrix}
        \begin{pmatrix}
            1 \\
            e^{-2 i \lambda j}
        \end{pmatrix}^T.
    \end{equation*}
    Thus, for $n \in \N$ and $1 \leq j_1 < \ldots < j_n \leq j-1$,
    \begin{equation*}
        N_{j_1} (\lambda) \cdots N_{j_n} (\lambda)
        = \frac{1}{(2 i \lambda)^n} \prod_{k = 1}^{n-1} \left( 1 - e^{2 i \lambda (j_{k+1} - j_k)} \right)
        \begin{pmatrix}
            1 \\
            -e^{i 2 \lambda j_1}
        \end{pmatrix}
        \begin{pmatrix}
            1 \\
            e^{-2 i \lambda j_n}
        \end{pmatrix}^T.
    \end{equation*}
    so that
    \begin{equation*}
        N_{j_1} (\lambda) \cdots N_{j_n} (\lambda) N_{j} (\lambda) (1, 0)^T
        = \frac{1}{(2 i \lambda)^{n+1}} \prod_{k = 1}^{n-1} \left( 1 - e^{2 i \lambda (j_{k+1} - j_k)} \right) \left( 1 - e^{2 i \lambda (j - j_n)} \right)
        \begin{pmatrix}
            1 \\
            -e^{2 i \lambda j_1}
        \end{pmatrix}
    \end{equation*}
    and
    \begin{equation*}
        \left\| N_{j_1} (\lambda) \cdots N_{j_n} (\lambda) N_{j} (\lambda) (1, 0)^T \right\|
        \leq \frac{2}{|\lambda|^{n+1}}
    \end{equation*}
    as $\operatorname{Im} (\lambda) > 0$, $j - j_n > 0$ and $j_{k+1} - j_k > 0$ for all $k = 1, \ldots, n-1$. Hence,
    \begin{multline*}
        \left\| \left( P_j (\lambda)^{-1} - P_{j-1} (\lambda)^{-1} \right) (1, 0)^T \right\| \\
        \leq \left| \alpha_j \right| \sum_{1 \leq n \leq j-1} \, \sum_{1 \leq j_1 < \ldots < j_n \leq j-1} \left| \alpha_{j_1} \cdots \alpha_{j_n} \right| \left\| N_{j_1} \cdots N_{j_n} N_j (1, 0)^T \right\| \\
        \leq \left| \frac{2 \alpha_j}{\lambda} \right| \sum_{1 \leq n \leq j-1} \, \sum_{1 \leq j_1 < \ldots < j_n \leq j-1} \left| \frac{\alpha_{j_1}}{\lambda} \right| \cdots \left| \frac{\alpha_{j_n}}{\lambda} \right|
        \leq \left| \frac{\alpha_j}{\lambda} \right| \exp \left( \frac{\| \alpha \|_{\ell^1 (\Z)}}{|\lambda|} \right)
    \end{multline*}
    Finally, we have
    \begin{equation*}
        \sum_{j \geq 2} \left\| \left( P_{j+1}^{-1} (\lambda) - P_{j}^{-1} (\lambda) \right) (1, 0)^T \right\|
        \leq \exp \left( \frac{\| \alpha \|_{\ell^1 (\Z)}}{|\lambda|} \right) \sum_{j \geq 2} \left| \frac{2 \alpha_j}{\lambda} \right|
        < \infty,
    \end{equation*}
    so that $(P_j (\lambda)^{-1} (1, 0)^T)_{j \geq 2}$ is a Cauchy sequence and therefore converges. Similarly, we prove that $P_{-\infty} (\lambda)^{-1} (0, 1)^T$ is well-defined. The holomorphic nature follows from the fact that $\lambda \mapsto N_j (\lambda)$ is holomorphic on the upper half plane.
    
    yThen, assume that $\lambda = 0$. By similar computation, we have
    \begin{equation*}
        \left\| N_{j_1} (\lambda) \cdots N_{j_n} (\lambda) N_{j} (\lambda) (1, 0)^T \right\|
        \lesssim |j_1| \left( \prod_{k = 1}^{n-1} (j_{k+1} - j_k) \right) (j - j_n)
        \leq j \prod_{k = 1}^n j_k
    \end{equation*}
    Hence,
    \begin{align*}
        \left\| \left( P_j (\lambda)^{-1} - P_{j-1} (\lambda)^{-1} \right) (1, 0)^T \right\|
        %
        %
        & \leq \left| j \alpha_j \right| \sum_{1 \leq n \leq j-1} \, \sum_{1 \leq j_1 < \ldots < j_n \leq j-1} \left| j_1 \alpha_{j_1} \right| \cdots \left| j_n \alpha_{j_n} \right| \\
        & \leq \left| j \alpha_j \right| \exp \left( \| \alpha \|_{\ell^{1, 1} (\Z)} \right).
    \end{align*}
    Finally, we have
    \begin{equation*}
        \sum_{j \geq 2} \left\| \left( P_{j+1}^{-1} (0) - P_{j}^{-1} (0) \right) (1, 0)^T \right\|
        \leq \| \alpha \|_{\ell^{1, 1} (\Z)} \exp \left( \| \alpha \|_{\ell^{1, 1} (\Z)} \right) \sum_{j \geq 2}
        < \infty,
    \end{equation*}
    so that $(P_j (\lambda)^{-1} (1, 0)^T)_{j \geq 2}$ is a Cauchy sequence and therefore converges. Similarly, we prove that $P_{-\infty} (\lambda)^{-1} (0, 1)^T$ is well-defined. Thus, Assertion $(2)$ is proved.
\end{proof}

\begin{remark}
    If $\lambda \in \C^+ \cup \{ 0 \}$, the matrices $P_{\pm \infty} (\lambda)^{-1}$ are not well-defined, unless $\alpha$ has exponential decay in the case $\lambda \in \C^+$ or $\alpha \in \ell^{1, 2} (\Z)$. We however denote the limits of $(P_j (\lambda)^{-1} (1, 0)^T)_{j \in \Z}, \, (P_j (\lambda)^{-1} (0, 1)^T)_{j \in \Z} \subset \C^2$ as $P_\infty^{-1} (\lambda) (1, 0)^T$ and $P_{-\infty}^{-1} (\lambda) (0, 1)^T$ for the seek of clarity.
\end{remark}

\begin{remark} \label{rkAss}
    Observe that, for $\lambda \neq 0$, the assumption $\alpha \in \ell^1(\Z)$ is sufficient to obtain the limit result. Only the case $\lambda = 0$ requires to assume $\alpha \in \ell^{1, 1}(\Z)$. The same observation holds for the definition of Jost solutions in Proposition \ref{prpJostSol} below.
\end{remark}

The next proposition states that the Jost solutions associated to $\H_{\alpha}$ are well-defined and holomorphic in $\lambda$ on $\C^+$.

\begin{proposition} \label{prpJostSol}
    Assume that $\alpha \in \ell^{1, 1} (\Z)$. Let $\lambda \in \C^+ \cup \R$. Then, the Jost solutions $\operatorname{f}_+ (\lambda, \cdot)$ and $\operatorname{f}_- (\lambda, \cdot)$ are well-defined. There exist coefficients $(A_j^\pm (\lambda), B_j^\pm (\lambda))_{j \in \Z} \subset \C^2$ such that the following assertions hold.
    \begin{enumerate}
        \item If $\lambda \in \R^* \cup \C^+$, then
        \begin{align*}
            & \left( A_j^+ (\lambda), B_j^+ (\lambda) e^{2 \operatorname{Im} (\lambda) j} \right) \to (1, 0) \text{ as } j \to + \infty, \\
            & \left( A_j^- (\lambda) e^{-2 \operatorname{Im} (\lambda) j}, B_j^- (\lambda) \right) \to (0, 1) \text{ as } j \to - \infty,
        \end{align*}
        and, for all $x \in \R$,
        \begin{equation} \label{eqJostSol14}
            \operatorname{f}_\pm (\lambda, x) = \sum_{j \in \Z} \left(A_j^\pm (\lambda) e^{i \lambda x} + B_j^\pm (\lambda) e^{-i \lambda x} \right) \mathbf{1}_{[j, j+1)} (x).
        \end{equation}
        Furthermore,
        \begin{enumerate}
            \item the function $\lambda \in \C^+ \mapsto (A_j^\pm (\lambda), B_j^\pm (\lambda))$ is holomorphic;
            \item the function $\lambda \in \R^* \mapsto (A_j^\pm (\lambda), B_j^\pm (\lambda))$ is $C^1$; and, for $\lambda \in \R^*$, $\operatorname{f}_\pm (\lambda, \cdot) \in L^\infty (\R)$.
        \end{enumerate}
        
        \item If $\lambda = 0$, then
        \begin{equation*}
            (A_j^\pm (0), B_j^\pm (0)) \to (1, 0) \text{ as } j \to \pm \infty,
        \end{equation*}
        and, for all $x \in \R$,
        \begin{equation} \label{eqJostSol15}
            \operatorname{f}_\pm (0, x) = \sum_{j \in \Z} \left( A_j^\pm (0) + B_j^\pm (0) x \right) \mathbf{1}_{[j, j+1)} (x).
        \end{equation}
        
        
    \end{enumerate}
\end{proposition}

\begin{proof}
    The proof is divided into two steps, each one proving an assertion. We proceed by analysis-synthesis for the construction of the solutions.

    \medskip
    
    \noindent \emph{Proof of $(1)$.} Let $\lambda \in \C^+ \cup \R^*$. We start by the analysis part of the argument. If $\operatorname{f}_\pm (\lambda, \cdot)$ satisfies \eqref{eqJost1}, then
    \begin{equation*}
        \operatorname{f}_\pm (\lambda, x) = \sum_{j \in \Z} \left( A_j^\pm e^{i \lambda x} + B_j^\pm e^{-i \lambda x} \right) \mathbf{1}_{[j, j+1)} (x)
    \end{equation*}
    for $x \in \R$ and for some $(A_j^\pm, B_j^\pm) = (A_j^\pm (\lambda), B_j^\pm (\lambda)) \subset \C^2$.
    
    If $\operatorname{f}_\pm (\lambda, \cdot)$ satisfy \eqref{eqJost2}-\eqref{eqJost3}, then, for $j \in \Z$,
    \begin{equation*}
        A_j^\pm e^{i \lambda j} + B_j^\pm e^{-i \lambda j}
        = A_{j-1}^\pm e^{i \lambda j} + B_{j-1}^\pm e^{-i \lambda j}
    \end{equation*}
    and
    \begin{equation*}
        \left( i \lambda A_j^\pm e^{i \lambda j} - i \lambda B_j^\pm e^{-i \lambda j} \right) - \left( i \lambda A_{j-1}^\pm e^{i \lambda j} - i \lambda B_{j-1}^\pm e^{-i \lambda j} \right)
        = \alpha_j \left( A_j^\pm e^{i \lambda j} + B_j^\pm e^{-i \lambda j} \right),
    \end{equation*}
    so that
    \begin{equation*}
        \begin{pmatrix}
            e^{i \lambda j} & e^{-i \lambda j} \\
            (i \lambda - \alpha_j) e^{i \lambda j} & (-i \lambda - \alpha_j) e^{-i \lambda j}
        \end{pmatrix}
        \begin{pmatrix}
            A_j^\pm \\
            B_j^\pm
        \end{pmatrix}
        =
        \begin{pmatrix}
            e^{i \lambda j} & e^{-i \lambda j} \\
            i \lambda e^{i \lambda j} & -i \lambda e^{-i \lambda j}
        \end{pmatrix}
        \begin{pmatrix}
            A_{j-1}^\pm \\
            B_{j-1}^\pm
        \end{pmatrix}.
    \end{equation*}
    Then,
    \begin{equation} \label{eqJostSol11}
        \begin{pmatrix}
            A_j^\pm \\
            B_j^\pm
        \end{pmatrix}
        =
        \begin{cases}
            M_j
            \begin{pmatrix}
                A_{j-1}^\pm \\
                B_{j-1}^\pm
            \end{pmatrix}
            & \text{ if } j \geq 1, \\
            M_j
            \begin{pmatrix}
                A_{j+1}^\pm \\
                B_{j+1}^\pm
            \end{pmatrix}
            & \text{ if } j \leq -1,
        \end{cases}
    \end{equation}
    where $M_j \in \operatorname{GL}_2 (\C)$ has been defined in \eqref{eqMatrix1}. In particular, $(A_j^\pm, B_j^\pm)$ is bounded in $\C^2$ by Lemma \ref{lemMatrix}.
    
    If $\operatorname{f}_+ (\lambda, \cdot)$ satisfies \eqref{eqJost4}, then
    \begin{equation} \label{eqJostSol16}
        \operatorname{f}_+ (\lambda, x) e^{-i \lambda x} - 1 \to 0
    \end{equation}
    as $x \to \infty$. Assume by contradiction that $|B_j^+| e^{2 \operatorname{Im} (\lambda) j}$ does not converge to $0$ as $j \to \infty$. For any $j \in \Z$ and $x \in [j, j+1]$, we have
    \begin{equation} \label{eqJostSol12}
        \operatorname{Re} \left( \operatorname{f}_+ (\lambda, x) e^{-i \lambda x} - 1 \right)
        = \left( \operatorname{Re} (A_j^+) - 1 \right) + |B_j^+| e^{2 \operatorname{Im} (\lambda) x} \cos \left( -2 \operatorname{Re} (\lambda) x + \operatorname{arg} (B_j^+) \right).
    \end{equation}
    Applying the above formula to $x = j$ and $x = j + 1/2$, subtracting both terms and passing to the limit using \eqref{eqJostSol16}, we get
    \begin{multline*}
        |B_j^+| e^{2 \operatorname{Im} (\lambda) j} \\
        \times \left( \cos \left( -2 \operatorname{Re} (\lambda) j + \operatorname{arg} (B_j^+) \right) - e^{\operatorname{Im} (\lambda)} \cos \left( -2 \operatorname{Re} (\lambda) j + \operatorname{arg} (B_j^+) - \operatorname{Re} (\lambda) \right) \right)
        \to 0 
    \end{multline*}
    as $j \to \infty$. In particular, there exists $R (\lambda) > 0$ and $\phi (\lambda) \in \R$
    \begin{multline*}
        \cos \left( -2 \operatorname{Re} (\lambda) j + \operatorname{arg} (B_j^+) \right) - e^{\operatorname{Im} (\lambda)} \cos \left( -2 \operatorname{Re} (\lambda) j + \operatorname{arg} (B_j^+) - \operatorname{Re} (\lambda) \right) \\
        = R (\lambda) \cos \left( -2 \operatorname{Re} (\lambda) j + \operatorname{arg} (B_j^+) - \phi (\lambda) \right)
        \to 0.
    \end{multline*}
    Thus,
    \begin{equation*}
        -2 \operatorname{Re} (\lambda) j + \operatorname{arg} (B_j^+) - \phi (\lambda)
        \to \frac{\pi}{2} \mod{\pi}.
    \end{equation*}
    Proceeding in the same way with $x = j + h$ and $x = j + 1/2 + h$, where $h < 1/8$, we obtain
    \begin{equation*}
        -2 \operatorname{Re} (\lambda) j + \operatorname{arg} (B_j^+)  - \phi (\lambda) + h
        \to \frac{\pi}{2} \mod{\pi}.
    \end{equation*}
    This is a contradiction. This implies that $|B_j^+| e^{\operatorname{Im} (\lambda) j} \to 0$ as $j \to \infty$, so that $B_j^+ \to 0$ (with exponential decay) and $\operatorname{Re} (A_j^+) \to 1$ as $j \to \infty$ by \eqref{eqJostSol16}. Similarly, we have that $\operatorname{Im} (A_j^+) \to 0$ as $j \to \infty$. Finally, $(A_j^+, B_j^+) \to (1, 0)$ as $j \to \infty$. Arguing as before, we obtain that, if $\operatorname{f}_- (\lambda, \cdot)$ satisfies \eqref{eqJost5}, then $(A_j^-, B_j^-) \to (0, 1)$ as $j \to -\infty$. In particular, the coefficients $(A_0^\pm, B_0^\pm)$ satisfy
    \begin{equation} \label{eqJostSol2}
        \begin{pmatrix}
            A_0^+ \\
            B_0^+
        \end{pmatrix}
        \coloneqq P_{\infty}^{-1}
        \begin{pmatrix}
            1 \\
            0
        \end{pmatrix}
        , \quad
        \begin{pmatrix}
            A_0^- \\
            B_0^-
        \end{pmatrix}
        \coloneqq P_{-\infty}^{-1}
        \begin{pmatrix}
            0 \\
            1
        \end{pmatrix},
    \end{equation}
    where $P_{\infty}^{-1} (1, 0)^T$ and $P_{-\infty}^{-1} (0, 1)^T$ have been defined in Lemma \ref{lemMatrix}. This concludes the analysis part of the argument.
    
    We now proceed with the synthesis part of the argument. Let $(A_0^\pm, B_0^\pm) \in \C^2$ given by \eqref{eqJostSol2}. Define $(A_j^\pm, B_j^\pm) \subset \C^2$ by
    \begin{equation} \label{eqJostSol3}
        \begin{pmatrix}
            A_j^\pm \\
            B_j^\pm
        \end{pmatrix}
        = P_j
        \begin{pmatrix}
            A_0^\pm \\
            B_0^\pm
        \end{pmatrix}
    \end{equation}
    for $j \in \Z^*$, where $P_j \in \operatorname{GL}_2 (\C)$ have been defined in Lemma \ref{lemMatrix}. Observe that \eqref{eqJostSol3} is equivalent to \eqref{eqJostSol11}. Finally, we define $\operatorname{f}_\pm (\lambda, \cdot)$ as in \eqref{eqJostSol14}. It is direct that $\operatorname{f}_\pm (\lambda, \cdot)$ satisfies \eqref{eqJost1}. Furthermore, $\operatorname{f}_\pm (\lambda, \cdot)$ satisfies \eqref{eqJost2} and \eqref{eqJost3} by \eqref{eqJostSol3}. Next, we have $(A_j^+, B_j^+) \to (1, 0)$ as $j \to \infty$ and $(A_j^-, B_j^-) \to (0, 1)$ as $j \to -\infty$ by \eqref{eqJostSol2} and Lemma \ref{lemMatrix}. Therefore, $\operatorname{f}_+ (\lambda, \cdot)$ satisfies \eqref{eqJost4} and $\operatorname{f}_- (\lambda, \cdot)$ satisfies \eqref{eqJost5}. We deduce that $\operatorname{f}_+ (\lambda, \cdot)$ and $\operatorname{f}_- (\lambda, \cdot)$ are the Jost solutions and this concludes the synthesis part of the argument.
    
    Let $\lambda \in \C^+$. As $P_\infty (\lambda)^{-1} (1, 0)^T$ and $P_{-\infty} (\lambda)^{-1} (0, 1)^T$ are holomorphic on $\C^+$ by Lemma \ref{lemMatrix}-$(2)$, we obtain that coefficients $(A_j^\pm (\lambda), B_j^\pm (\lambda))$ are holomorphic on $\C^+$ by \eqref{eqJostSol3} and the fact that $(P_j (\lambda))$ are holomorphic on $\C^+$. This proves $(a)$.
    
    Let $\lambda \in \R^*$. As $P_\infty (\lambda)^{-1}$ and $P_{-\infty} (\lambda)^{-1}$ are $\C^1$ on $\R^*$ by Lemma \ref{lemMatrix}-$(1)$, we obtain that coefficients $(A_j^\pm (\lambda), B_j^\pm (\lambda))$ are $C^1$ on $\R^*$ by \eqref{eqJostSol3} and the fact that $(P_j (\lambda))$ are $C^1$ on $\R^*$. Furthermore, we have that $(P_j (\lambda))$ is bounded in $\operatorname{GL}_2 (\C)$ as it is a convergent sequence by Lemma \ref{lemMatrix}. In particular, coefficients $(A^\pm_j (\lambda), B^\pm_j (\lambda))$ are bounded in $\R^*$ by \eqref{eqJostSol3}. Thus, $\operatorname{f}_\pm (\lambda, \cdot) \in L^\infty(\R)$ as they are explicitly given by \eqref{eqJostSol14}. This proves $(b)$.
    
    \medskip
    
    \noindent \emph{Proof of $(2)$.} Let $\lambda = 0$. The reasoning is similar to the first part of the proof, so certain details are omitted for brevity. We start by the analysis part of the argument. If $\operatorname{f}_\pm (0, \cdot)$ satisfies \eqref{eqJost1}, then there exists $(A_j^\pm, B_j^\pm) \subset \C^2$ such that, for $x \in \R$,
    \begin{equation*}
        \operatorname{f}_\pm (0, x) = \sum_{j = -\infty}^\infty \left( A_j^\pm + B_j^\pm x \right) \mathbf{1}_{[j, j+1)} (x).
    \end{equation*}

    If $\operatorname{f}_\pm (0, \cdot)$ satisfies \eqref{eqJost2} and \eqref{eqJost3}, then, for $j \in \Z$, \eqref{eqJostSol11} is satisfied, where $M_j = M_j (0) \in \operatorname{GL}_2 (\C)$ has been defined in \eqref{eqMatrix1}.

    If $\operatorname{f}_+ (0, \cdot)$ satisfies \eqref{eqJost4}, then
    \begin{equation*}
        \sum_{j \in \Z} \left( A_j^+ + B_j^+ x \right) \mathbf{1}_{[j, j+1)} (x) 
        \to 1
    \end{equation*}
    as $x \to \infty$. Let $h \in \R$ be small and $(x_j)_{j \in \Z} \subset \R$ such that $x_j$ and $x_j + h \in (j, j+1)$. Taking the real part of the above expression, we have
    \begin{equation*}
        \operatorname{Re} (A_j^+) - 1 + \operatorname{Re} (B_j^+) x_j
        \to 0
    \end{equation*}
    and
    \begin{equation*}
        \operatorname{Re} (A_j^+) - 1 + \operatorname{Re} (B_j^+) (x_j + h)
        \to 0
    \end{equation*}
    as $j \to \infty$. As $h$ has been taken arbitrary, we obtain that $\operatorname{Re} (B_j^+) \to 0$ as $j \to \infty$. Applying the same reasoning to the imaginary part yields $(A_j^+, B_j^+) \to (1, 0)$ as $j \to \infty$. Similarly, if $\operatorname{f}_- (0, \cdot)$ satisfies \eqref{eqJost5}, then $(A_j^-, B_j^-) \to (1, 0)$ as $j \to -\infty$. This concludes the analysis part of the argument.

    We now proceed with the synthesis part of the argument. We define $(A_0^\pm, B_0^\pm) \in \C^2$ as
    \begin{equation*}
        \begin{pmatrix}
            A_0^+ \\
            B_0^+
        \end{pmatrix}
        \coloneqq P_{\infty}^{-1}
        \begin{pmatrix}
            1 \\
            0
        \end{pmatrix}
        , \quad
        \begin{pmatrix}
            A_0^- \\
            B_0^-
        \end{pmatrix}
        \coloneqq P_{-\infty}^{-1}
        \begin{pmatrix}
            1 \\
            0
        \end{pmatrix},
    \end{equation*}
    where $P_{\infty}^{-1} (1, 0)^T$ and $P_{-\infty}^{-1} (1, 0)^T$ have been defined in Lemma \ref{lemMatrix}-$(2)$; the sequences $(A_j^\pm)$ and $(B_j^\pm)$ as in \eqref{eqJostSol3}, where $P_j = P_j (0) \in \operatorname{GL}_2 (\C)$ have been defined in Lemma \ref{lemMatrix}; and $\operatorname{f}_\pm (0, \cdot)$ as in \eqref{eqJostSol15}. The remainder of the synthesis follows the same steps as in the proof of $(1)$ and we deduce that $\operatorname{f}_+ (0, \cdot)$ and $\operatorname{f}_- (0, \cdot)$ are the Jost solutions. This concludes the proof.
\end{proof}

Let $f, g$ be solutions of \eqref{eqJost1}. For $x \in \R \setminus \Z$, we define the Wronskian $W (f (x), g (x))$ as
\begin{equation*}
    W \left( f (x), g (x) \right)
    = f (x) \partial_x g (x) - \partial_x f (x) g (x).
\end{equation*}
Observe that the Wronskian is constant on each interval $(j, j+1)$ but depends \textit{a priori} of $j$. The following lemma states that it is constant on $\R \setminus \Z$ if $f$ and $g$ satisfy continuity on $\R$ and \eqref{eqJost3}.

\begin{lemma} \label{lemWskDelta}
    Let $\lambda \in \R \cup \C^+$ and $f, g$ be solutions of \eqref{eqJost1} which both satisfy \eqref{eqJost2} and \eqref{eqJost3}. Then, $x \mapsto W (f (x), g (x))$ is constant on $\R \setminus \Z$.
\end{lemma}

\begin{proof}
    First, observe that there exists $(A_j, B_j, C_j, D_j)_{j \in \Z} \subset \C^4$ such that, if $\lambda \neq 0$,
    \begin{align*}
        f (x) = \sum_{j \in \Z} \left( A_j e^{i \lambda x} + B_j e^{-i \lambda x} \right) \boldsymbol{1}_{[j, j+1)} (x), \quad
        g (x) = \sum_{j \in \Z} \left( C_j e^{i \lambda x} + D_j e^{-i \lambda x} \right) \boldsymbol{1}_{[j, j+1)} (x)
    \end{align*}
    and, if $\lambda = 0$,
    \begin{align*}
        f (x) = \sum_{j \in \Z} \left( A_j + B_j x \right) \boldsymbol{1}_{[j, j+1)} (x), \quad
        g (x) = \sum_{j \in \Z} \left( C_j + D_j x \right) \boldsymbol{1}_{[j, j+1)} (x).
    \end{align*}
    Furthermore, for $j \in \Z$ and $x \in (j, j+1)$, we have
    \begin{equation} \label{eqWskDelta}
        W \left( f (x), g (x) \right)
        =
        \begin{cases}
            -2 i \lambda \left( A_j D_j - B_j C_j \right) & \text{ if } \lambda \neq 0, \\
            A_j D_j - B_j C_j & \text{ if } \lambda = 0,
        \end{cases}
    \end{equation}
    so the Wronskian is piece-wise constant. For $j \geq 1$, we have
    \begin{equation*}
        M_j^T
        \begin{pmatrix}
            0 & 1 \\
            -1 & 0
        \end{pmatrix}
        M_j
        = \det (M_j)
        \begin{pmatrix}
            0 & 1 \\
            -1 & 0
        \end{pmatrix}
        =
        \begin{pmatrix}
            0 & 1 \\
            -1 & 0
        \end{pmatrix}
    \end{equation*}
    where $M_j$ has been defined in \eqref{eqMatrix1}, and
    \begin{equation*}
        \begin{pmatrix}
            A_j \\
            B_j
        \end{pmatrix}
        = M_j
        \begin{pmatrix}
            A_{j-1} \\
            B_{j-1}
        \end{pmatrix},
        \quad
        \begin{pmatrix}
            C_j \\
            D_j
        \end{pmatrix}
        = M_j
        \begin{pmatrix}
            C_{j-1} \\
            D_{j-1}
        \end{pmatrix}.
    \end{equation*}
    by continuity of $f$ and $g$ on $\R$ and the fact that they satisfy \eqref{eqJost3}. Thus,
    \begin{align*}
        A_j D_j - B_j C_j
        & = (A_j \, B_j)
        \begin{pmatrix}
            0 & 1 \\
            -1 & 0
        \end{pmatrix}
        \begin{pmatrix}
            C_j \\
            D_j
        \end{pmatrix} \\
        & = (A_{j-1} \, B_{j-1}) M_j^T
        \begin{pmatrix}
            0 & 1 \\
            -1 & 0
        \end{pmatrix}
        M_j
        \begin{pmatrix}
            C_{j-1} \\
            D_{j-1}
        \end{pmatrix} \\
        & = (A_{j-1} \, B_{j-1})
        \begin{pmatrix}
            0 & 1 \\
            -1 & 0
        \end{pmatrix}
        \begin{pmatrix}
            C_{j-1} \\
            D_{j-1}
        \end{pmatrix} \\
        & = A_{j-1} D_{j-1} - B_{j-1} C_{j-1}.
    \end{align*}
    Thus, the Wronskian \eqref{eqWskDelta} is constant on $x \in \R^+$. A similar computation holds for $j \leq -1$. This concludes the proof.
\end{proof}

In what follows, we shall write $W (f (x), g (x))$ of two functions $f, g$ simply as $W (f, g)$ if it does not depend on $x$.

The Wronskian of the Jost solutions $\operatorname{W} (\lambda)$ is given by
\begin{align} \label{eqW}
    \operatorname{W} (\lambda)
    & = \operatorname{f}_+ (\lambda, \cdot) \partial_x \operatorname{f}_- (\lambda, \cdot) -  \partial_x \operatorname{f}_+ (\lambda, \cdot) \operatorname{f}_- (\lambda, \cdot) \\
    & =
    \begin{cases}
        -2 i \lambda \left( A_j^+ (\lambda) B_j^- (\lambda) - B_j^+ (\lambda) A_j^- (\lambda) \right) & \text{ if } \lambda \neq 0, \\
        A_j^+ (0) B_j^- (0) - B_j^+ (0) A_j^- (0) & \text{ if } \lambda = 0,
    \end{cases}
\end{align}
for any $j \in \Z$, and does not depends on $j$ by Lemma \ref{lemWskDelta}.

Let $\lambda \in \C^+ \cup \R^*$. As $\operatorname{f}_+ (\lambda, \cdot)$ and $\operatorname{f}_+ (-\lambda, \cdot)$ both satisfy \eqref{eqJost2} and \eqref{eqJost3}, their Wronskian $W (\operatorname{f}_+ (\lambda, \cdot), \operatorname{f}_+ (-\lambda, \cdot))$ is constant by Lemma \ref{lemWskDelta}. Using their explicit form given by Proposition \ref{prpJostSol}, we have
\begin{align*}
    W \left( \operatorname{f}_+ (\lambda, \cdot), \operatorname{f}_+ (-\lambda, \cdot) \right)
    = - 2 i \lambda \left( |A_j^+|^2 - |B_j^+|^2 \right)
    \to - 2 i \lambda
    \neq 0.
\end{align*}
as $j \to \infty$. Thus, $\operatorname{f}_+ (\lambda, \cdot)$ and $\operatorname{f}_+ (-\lambda, \cdot)$ are linearly independent on $\R$. Therefore they form a basis of solutions to the second order equation \eqref{eqJost1}. Thus, there exist $\operatorname{a}_\pm (\lambda), \operatorname{b}_\pm (\lambda) \in \C$ such that
\begin{align}
    \operatorname{f}_- (\lambda, \cdot)
    & = \operatorname{a}_- (\lambda) \operatorname{f}_+ (\lambda, \cdot) + \operatorname{b}_- (\lambda) \operatorname{f}_+ (-\lambda, \cdot), \label{eqCoeff1} \\
    \operatorname{f}_+ (\lambda, \cdot)
    & = \operatorname{a}_+ (\lambda) \operatorname{f}_- (\lambda, \cdot) + \operatorname{b}_+ (\lambda) \operatorname{f}_- (-\lambda, \cdot). \label{eqCoeff2}
\end{align}
In particular, we have
\begin{equation*}
    \operatorname{W} (\lambda)
    = -2 i \lambda \operatorname{b}_- (\lambda)
    = -2 i \lambda \operatorname{b}_+ (\lambda),
\end{equation*}
so that
\begin{equation*}
    \operatorname{b}_+ (\lambda)
    = \operatorname{b}_- (\lambda)
    = \operatorname{b} (\lambda)
\end{equation*}
and $\operatorname{W} (\lambda) = 0$ if and only if $\operatorname{f}_+ (\lambda, \cdot)$ and $\operatorname{f}_- (\lambda, \cdot)$ are linearly dependent on $\R$.

The following result states that $\operatorname{W} (\lambda)$ is non-zero on $\C^+ \cup \R^*$. It has been proven in \cite{Fa64} in the case of a potential $V$ which is a function.

\begin{proposition} \label{prpWronski}
    For all $\lambda \in \C^+ \cup \R^*$, we have $|\operatorname{b} (\lambda)| \geq 1$. In particular, $\operatorname{W} (\lambda) \neq 0$.
\end{proposition}

\begin{proof}
    Let $\lambda \in \C^+ \cup \R^*$. Using the explicit form of $\operatorname{f}_- (\lambda, \cdot)$ and $\overline{\operatorname{f}_- (\lambda, \cdot)}$ given by Proposition \ref{prpJostSol} and taking $j \to -\infty$, we have
    \begin{align*}
        W \left( \operatorname{f}_- (\lambda, x), \overline{\operatorname{f}_- (\lambda, x)} \right)
        & = 2 i \lambda \left( |A_j^- (\lambda)|^2 - |B_j^- (\lambda)|^2 \right) \\
        & = -2 i \lambda.
    \end{align*}
    Furthermore, using \eqref{eqCoeff1}, we have
    \begin{align*}
        W \left( \operatorname{f}_- (\lambda, x), \overline{\operatorname{f}_- (\lambda, x)} \right)
        & = 2 i \lambda \left( |\operatorname{a} (\lambda)|^2 - |\operatorname{b}_- (\lambda)|^2 \right)
        \left( |A_j^+ (\lambda)|^2 - |B_j^+ (\lambda)|^2 \right) \\
        & \to 2 i \lambda \left( |\operatorname{a} (\lambda)|^2 - |\operatorname{b}_- (\lambda)|^2 \right)
    \end{align*}
    as $j \to \infty$. However, $\operatorname{f}_- (\lambda, \cdot)$ and $\overline{\operatorname{f}_- (\lambda, \cdot)}$ satisfy \eqref{eqJost2} and \eqref{eqJost3}, hence their Wronskian is also constant by Lemma \ref{lemWskDelta}. Thus, combining the two values of $W ( \operatorname{f}_- (\lambda, x), \overline{\operatorname{f}_- (\lambda, x)} )$ obtained previously, we get
    \begin{equation*}
        |\operatorname{b} (\lambda)|^2 = 1 + |\operatorname{a}_- (\lambda)|^2,
    \end{equation*}
    so that $\operatorname{b} (\lambda) \neq 0$. Thus, by \eqref{eqCoeff1}, we obtain that $\operatorname{f}_+ (\lambda, \cdot)$ and $\operatorname{f}_- (\lambda, \cdot)$ are linearly independent on $\R$, so that $\operatorname{W} (\lambda) \neq 0$. This concludes the proof.
\end{proof}

\subsection{Properties of the Wronskian under stronger decay assumptions}

This subsection is devoted to prove the next proposition, which will be used in the proof of the low-energy part of the dispersive estimate (see Proposition \ref{prpLowEn}). An analog result was proved in \cite[Lemma $1$]{DeTr79} in the case of a potential $V$ which is a function.

\begin{proposition} \label{prpNonZero}
    Assume that $\alpha \in \ell^{1, 2} (\Z)$. Then, $\lambda \in \R \mapsto \operatorname{W} (\lambda) \lambda^{-1}$ is continuous and never vanishes.
\end{proposition}

We define $\operatorname{m}_\pm : (\C^+ \cup \R) \times \R \to \C$ as
\begin{equation} \label{eqm+-}
    \operatorname{m}_\pm (\lambda, x)
    = e^{\mp i \lambda x} \operatorname{f}_\pm (\lambda, x),
\end{equation}
so that
\begin{align}
    \operatorname{m}_+ (\lambda, x)
    & = \sum_{j \in \Z} \mathbf{1}_{[j, j+1)} (x) 
    \begin{cases}
        A_j^+ (\lambda) + B_j^+ (\lambda) e^{-2 i \lambda x} & \text{ if } \lambda \neq 0, \\
        A_j^+ (0) + B_j^+ (0) x & \text{ if } \lambda = 0,
    \end{cases} \label{eqM+1}
    \\
    \operatorname{m}_- (\lambda, x)
    & = \sum_{j \in \Z} \mathbf{1}_{[j, j+1)} (x) 
    \begin{cases}
        A_j^- (\lambda) e^{2 i \lambda x} + B_j^- (\lambda) & \text{ if } \lambda \neq 0, \\
        A_j^- (0) + B_j^- (0) x & \text{ if } \lambda = 0.
    \end{cases} \label{eqM-1}
\end{align}

Observe that $\operatorname{m}_\pm (\lambda, \cdot)$ are solutions of the equation
\begin{align}
    -\partial_{xx} \operatorname{m}_\pm (\lambda, x)
    & = \pm 2 i \lambda \partial_x \operatorname{m}_\pm (\lambda, x), \quad x \in \R \setminus \Z, \label{eqM1} \\
    \operatorname{m}_\pm (\lambda, j-)
    & = \operatorname{m}_\pm (\lambda, j+), \quad j \in \Z, \label{eqM2} \\
    \partial_x \operatorname{m}_\pm (\lambda, x) \big|_{x = j+} - \partial_x \operatorname{m}_\pm (\lambda, x) \big|_{x = j-}
    & = \alpha_j \operatorname{m}_\pm (\lambda, j), \quad j \in \Z, \label{eqM3} \\
    \operatorname{m}_\pm (\lambda, x)
    & \sim 1 \text{ as } x \to \pm \infty. \label{eqM4}
\end{align}
For $j \in \Z$ and $x \in (j-1, j)$, we have by \eqref{eqM1} that
\begin{equation*}
    \partial_x \left( e^{2 i \lambda x} \partial_x \operatorname{m}_+ (\lambda, x) \right)
    = 0
\end{equation*}
so that $x \mapsto e^{2 i \lambda x} \partial_x \operatorname{m}_+ (\lambda, x)$ is constant on each interval $(j-1, j)$.
Thus, by \eqref{eqM3}, we get
\begin{align*}
    e^{2 i \lambda x} \partial_x \operatorname{m}_+ (\lambda, x)
    & = e^{2 i \lambda j} \partial_x \operatorname{m}_+ (\lambda, j-) \\
    & = e^{2 i \lambda j} \partial_x \operatorname{m}_+ (\lambda, j+) + e^{2 i \lambda j} \alpha_j \operatorname{m}_+ (\lambda, j),
\end{align*}
where $\operatorname{m}_+ (\lambda, j)$ is well-defined by \eqref{eqM2}. Iterating the previous expression, we deduce that
\begin{align*}
    \partial_x \operatorname{m}_+ (\lambda, x)
    & = -e^{2 i \lambda x} \sum_{j > x} e^{2 i \lambda j} \alpha_j \operatorname{m}_+ (\lambda, j) \\
    & = -e^{2 i \lambda x} \sum_{j > x} \alpha_j \left( e^{2 i \lambda j} A_j^+ (\lambda) + B_j^+ (\lambda) \right),
\end{align*}
the last equality justifying the convergence of the serie by Proposition \ref{prpJostSol}.
By \eqref{eqM4}, we obtain that
\begin{align*}
    \operatorname{m}_+ (\lambda, x) - 1
    & = - \int_x^\infty \partial_s \operatorname{m}_+ (\lambda, s) \, ds \\
    & = \int_x^\infty e^{-2 i \lambda s} \sum_{j > s} e^{2 i \lambda j} \alpha_j \operatorname{m}_+ (\lambda, j) \, ds
\end{align*}
Similarly, doing similar computations for $\operatorname{m}_+ (\lambda, \cdot)$, we can prove that
\begin{align} 
    \operatorname{m}_+ (\lambda, x) - 1
    & = \frac{1}{2 i \lambda} \sum_{j > x} \alpha_j \left( e^{2 i \lambda (j - x)} - 1 \right) \operatorname{m}_+ (\lambda, j), \label{eqM+2} \\
    \operatorname{m}_- (\lambda, x) - 1
    & = \frac{1}{2 i \lambda} \sum_{j < x} \alpha_j \left( e^{-2 i \lambda (j - x)} - 1 \right) \operatorname{m}_- (\lambda, j). \label{eqM-2}
\end{align}
By iterating the formula \eqref{eqM+2}, we obtain that
\begin{equation} \label{eqm+3}
    \operatorname{m}_+ (\lambda, x) - 1
    = \sum_{n = 1}^\infty \sum_{x < j_1 < \ldots < j_n} \prod_{l = 1}^n \alpha_{j_l} \frac{e^{2 i \lambda (j_{l} - j_{l-1})} - 1}{2 i \lambda}
\end{equation}
where $j_0 = x$.

\begin{remark} \label{rkComput}
    If $\lambda \in \R^*$, we can re-write $\operatorname{m}_+$ as
    \begin{equation} \label{eqComput1}
        \operatorname{m}_+ (\lambda, x) - 1
        = \sum_{j \in \Z} \frac{\alpha_j}{2 i \lambda} \left( e^{-2 i \lambda (j - x)} - 1 \right) \operatorname{m}_+ (\lambda, j) - \sum_{j \leq x} \frac{\alpha_j}{2 i \lambda} \left( e^{-2 i \lambda (j - x)} - 1 \right) \operatorname{m}_+ (\lambda, j).
    \end{equation}
    Indeed, the sums are well-defined as $(A_j^+ (\lambda))$ and $(B_j^+ (\lambda))$ are bounded. If $\lambda \in \C^+$, the previous formulation may a priori be not well-defined as the coefficients $A_j^+ (\lambda)$ and $B_j^+ (\lambda)$ may diverge as $j \to -\infty$. Lemma \ref{lemBound} addresses this problem.
\end{remark}

The following lemma provides bounds on $\operatorname{m}_\pm (\cdot, x)$ and $\partial_\lambda \operatorname{m}_\pm (\cdot, x)$ on $\C^+$.

\begin{lemma} \label{lemBound}
    Let $s \in \R$. The following statements hold.
    \begin{enumerate}
        \item Assume that $\alpha \in \ell^{1, 1} (\Z)$. Then,
        \begin{equation} \label{eqBound1}
            \sup_{\lambda \in \C^+} |\operatorname{m}_+ (\lambda, s)| \lesssim 1 + \max(-s, 0), \quad
            \sup_{\lambda \in \C^+} |\operatorname{m}_- (\lambda, s)| \lesssim 1 + \max(s, 0)
        \end{equation}
        \item Assume that $\alpha \in \ell^{1, 2} (\Z)$. Then,
        \begin{equation} \label{eqBound2}
            \sup_{\lambda \in \C^+} |\partial_\lambda \operatorname{m}_+ (\lambda, s)| \lesssim 1 - s \max (-s, 0), \quad
            \sup_{\lambda \in \C^+} |\partial_\lambda \operatorname{m}_- (\lambda, s)| \lesssim 1 + s \max (s, 0).
        \end{equation}
    \end{enumerate}
\end{lemma}

\begin{proof}
    Let $\lambda \in \C^+$ be fixed. We divide the proof into two steps, each one showing an assertion.

    \medskip
    
    \noindent \emph{Proof of $(1)$.}
    Let $x \in \R$ and $j \in \Z$, we have
    \begin{equation*}
        \frac{e^{2 i \lambda (j - x)} - 1}{2 i \lambda}
        = \int_0^{j-x} e^{2 i \lambda s} \, ds
    \end{equation*}
    so that
    \begin{equation} \label{eqBound4}
        \left| \frac{e^{2 i \lambda (j - x)} - 1}{2 i \lambda} \right|
        \leq \int_0^{j-x} \left| e^{2 i \lambda s} \right| \, ds
        = \frac{1 - e^{-2 \operatorname{Im} (\lambda) (j-x)}}{2 \operatorname{Im} (\lambda)}
        \leq j - x.
    \end{equation}
    Thus,
    \begin{equation*}
        |\operatorname{m}_+ (\lambda, x)|
        \leq 1 + \sum_{j > x} (j - x) |\alpha_j| |\operatorname{m}_+ (\lambda, j)|.
    \end{equation*}
    
    Let $x \geq 0$. Then,
    \begin{equation*}
        |\operatorname{m}_+ (\lambda, x)|
        \leq 1 + \sum_{j > x} j |\alpha_j| |\operatorname{m}_+ (\lambda, j)|.
    \end{equation*}
    Thus,
    \begin{equation*}
        \sup_{y \geq x} \left| \operatorname{m}_+ (\lambda, y) \right|
        \leq 1 + \sup_{y \geq x} \left| \operatorname{m}_+ (\lambda, y) \right| \sum_{j > x} j |\alpha_j|.
    \end{equation*}
    As $\alpha \in \ell^{1, 1} (\Z)$, there exists $x_0$ large enough such that, if $x \geq x_0$, then
    \begin{equation*}
        \sup_{y \geq x} \left| \operatorname{m}_+ (\lambda, y) \right|
        \lesssim 1 + \frac{1}{2} \sup_{y \geq x} \left| \operatorname{m}_+ (\lambda, y) \right|,
    \end{equation*}
    so that
    \begin{equation} \label{eqBound11}
        \sup_{y \geq x} \left| \operatorname{m}_+ (\lambda, y) \right|
        \lesssim 1.
    \end{equation}
    Furthermore, if $0 \leq x < x_0$, then
    \begin{align*}
        \left| \operatorname{m}_+ (\lambda, x) \right|
        & \lesssim 1 + \sum_{x \leq j < x_0} j |\alpha_j| |\operatorname{m}_+ (\lambda, j)| + \sum_{j \geq x_0} j |\alpha_j| |\operatorname{m}_+ (\lambda, j)| \\
        & \lesssim 1 + \sum_{x \leq j < x_0} j |\alpha_j| |\operatorname{m}_+ (\lambda, j)| \\
        & \lesssim 1
    \end{align*}
    where we used \eqref{eqBound11} for the second inequality and the fact that the sum contains only finitely many indices for the third inequality. We thus obtain that,
    \begin{equation} \label{eqBound6}
        |\operatorname{m}_+ (\lambda, x)|
        \lesssim 1 \text{ for } x \geq 0.
    \end{equation}
    
    If $x < 0$, we have
    \begin{align*}
         |\operatorname{m}_+ (\lambda, x)|
         & \lesssim 1 + \sum_{j \geq 0} j |\alpha_j| |\operatorname{m}_+ (\lambda, j)| + |x| \sum_{j > x} |\alpha_j| |\operatorname{m}_+ (\lambda, j)| \\
         & \lesssim 1 + |x| \sum_{j > x} |\alpha_j| |\operatorname{m}_+ (\lambda, j)|,
    \end{align*}
    where the second inequality comes from \eqref{eqBound6}. Thus,
    \begin{equation*}
        \frac{|\operatorname{m}_+ (\lambda, x)|}{1 + |x|}
        \lesssim 1 + \sum_{j > x} |\alpha_j| (1 + |j|) \frac{|\operatorname{m}_+ (\lambda, j)|}{1 + |j|}
    \end{equation*}
    and we conclude by decomposing the sum, similarly as we did for the case $x \geq 0$. Thus,
    \begin{equation} \label{eqBound12}
        |\operatorname{m}_+ (\lambda, x)|
        \lesssim 1 - x \text{ for } x \leq 0.
    \end{equation}
    
    Combining \eqref{eqBound6} and \eqref{eqBound12}, we obtain that \eqref{eqBound1} holds for $\operatorname{m}_+ (\cdot, x)$. The bound for $\operatorname{m}_- (\cdot, x)$ is proved the same way.

    \medskip

    \noindent \emph{Proof of $(2)$.} For $h > 0$ and $x \in \R$, we define
    \begin{align*}
        D_h \operatorname{m}_+ (\lambda, x)
        & = \frac{\operatorname{m}_+ (\lambda + h, x) - \operatorname{m}_+ (\lambda, x)}{h} \\
        & = S_1 (\lambda, x, h) + S_2 (\lambda, x, h)
    \end{align*}
    with
    \begin{align*}
        S_1 (\lambda, x, h)
        & = \sum_{j > x} \frac{1}{h} \left( \frac{e^{2 i (\lambda + h) (j - x)}}{2 i (\lambda + h)} - \frac{e^{2 i \lambda (j - x)}}{2 i \lambda} \right) \alpha_j \operatorname{m}_+ (\lambda + h, j), \\
        S_2 (\lambda, x, h)
        & = \sum_{j > x} \frac{e^{2 i \lambda (j - x)} - 1}{2 i \lambda} \alpha_j D_h \operatorname{m}_+ (\lambda, j).
    \end{align*}
    Furthermore, we have
    \begin{equation} \label{eqBound8}
        \frac{1}{h} \left| \frac{e^{2 i (\lambda + h) (j - x)}}{2 i (\lambda + h)} - \frac{e^{2 i \lambda (j - x)}}{2 i \lambda} \right|
        \lesssim (j - x)^2.
    \end{equation}

    If $x \geq 0$, then
    \begin{equation} \label{eqBound7}
        |S_1 (\lambda, x, h)|
        \lesssim \sum_{j > x} (j - x)^2 |\alpha_j| |\operatorname{m}_+ (\lambda, j)|
        \lesssim \sum_{j > x} j^2 |\alpha_j| |\operatorname{m}_+ (\lambda, j)|
        \lesssim 1
    \end{equation}
    by \eqref{eqBound1}, \eqref{eqBound8} and the fact that $\alpha \in \ell^{1, 2} (\Z)$. Thus, by \eqref{eqBound7} and \eqref{eqBound4}, we obtain
    \begin{equation*}
        \left| D_h \operatorname{m}_+ (\lambda, x) \right|
        \lesssim 1 + \sum_{j > x} (j - x) |\alpha_j| \left| D_h \operatorname{m}_+ (\lambda, j) \right|
        \lesssim 1 + \sum_{j > x} j |\alpha_j| \left| D_h \operatorname{m}_+ (\lambda, j) \right|.
    \end{equation*}
    Thus,
    \begin{equation} \label{eqBound9}
        \left| D_h \operatorname{m}_+ (\lambda, x) \right|
        \lesssim 1 \text{ for } x \geq 0
    \end{equation}
    by an argument similar to the first part of the proof.
    
    If $x < 0$, we have
    \begin{align*}
        \sum_{j > x} j^2 |\alpha_j| |\operatorname{m}_+ (\lambda, j)|
        & = \sum_{j \geq 0} j^2 |\alpha_j| |\operatorname{m}_+ (\lambda, j)| + \sum_{j = x}^0 j^2 |\alpha_j| |\operatorname{m}_+ (\lambda, j)| \\
        & \leq \sum_{j \geq 0} j^2 |\alpha_j| |\operatorname{m}_+ (\lambda, j)| + x^2 \sum_{j = x}^0 |\alpha_j| |\operatorname{m}_+ (\lambda, j)| \\
        & \lesssim 1 + x^2
    \end{align*}
    and
    \begin{align*}
        \sum_{j > x} x^2 |\alpha_j| |\operatorname{m}_+ (\lambda, j)|
        & = x^2 \left( \sum_{j \geq 0} |\alpha_j| |\operatorname{m}_+ (\lambda, j)| + \sum_{j = x}^0 |\alpha_j| |\operatorname{m}_+ (\lambda, j)| \right) \\
        & \lesssim 1 + x^2
    \end{align*}
    by \eqref{eqBound1} and the fact that $\alpha \in \ell^{1, 2} (\Z)$. Thus,
    \begin{equation*}
        |S_1 (\lambda, x, h)|
        \lesssim \sum_{j > x} (j - x)^2 |\alpha_j| |\operatorname{m}_+ (\lambda, j)|
        \lesssim 1 + x^2.
    \end{equation*}
    Furthermore,
    \begin{align*}
        |S_2 (\lambda, x, h)|
        & \lesssim \sum_{j > x} (j - x) |\alpha_j| |D_h \operatorname{m}_+ (\lambda, j)| \\
        & \lesssim \sum_{j > x} j |\alpha_j| |D_h \operatorname{m}_+ (\lambda, j)| + |x| \sum_{j > x} |\alpha_j| |D_h \operatorname{m}_+ (\lambda, j)| \\
        & \lesssim \sum_{j \geq 0} j |\alpha_j| |D_h \operatorname{m}_+ (\lambda, j)| + |x| \sum_{j > x} |\alpha_j| |D_h \operatorname{m}_+ (\lambda, j)| \\
        & \lesssim 1 + |x| \sum_{j > x} |\alpha_j| |D_h \operatorname{m}_+ (\lambda, j)|
    \end{align*}
    where we used \eqref{eqBound4} for the first inequality and \eqref{eqBound9} for the last one. Thus,
    \begin{align*}
        |D_h \operatorname{m}_+ (\lambda, x)|
        & \lesssim |S_1 (\lambda, x, h)| + |S_2 (\lambda, x, h)| \\
        & \lesssim 1 + x^2 + |x| \sum_{j > x} |\alpha_j| |D_h \operatorname{m}_+ (\lambda, j)|,
    \end{align*}
    so that
    \begin{equation*}
        \frac{|D_h \operatorname{m}_+ (\lambda, x)|}{1 + x^2}
        \lesssim 1 + \sum_{j > x} (1 + |j|^2) \frac{|D_h \operatorname{m}_+ (\lambda, j)|}{1 + j^2}
    \end{equation*}
    and we conclude that
    \begin{equation} \label{eqBound10}
        |D_h \operatorname{m}_+ (\lambda, x)|
        \lesssim 1 + x^2 \text{ for } x < 0
    \end{equation}
    by an argument similar to the first part of the proof.
    
    Passing to the limit $h \to 0$ in \eqref{eqBound9} and \eqref{eqBound10}, we obtain \eqref{eqBound2} for $\partial_\lambda \operatorname{m}_+ (\cdot, x)$. The bound for $\partial_\lambda \operatorname{m}_- (\cdot, x)$ is proved the same way. This concludes the proof.
\end{proof}

We deduce the proof of Proposition \ref{prpNonZero}.

\begin{proof}[Proof of Proposition \ref{prpNonZero}]
    Let $\lambda \in \C^+ \cup \R^*$ and $x \in \R$. We recall that
    \begin{equation*}
        \frac{\operatorname{W} (\lambda)}{\lambda}
        = -2 i \operatorname{b} (\lambda).
    \end{equation*}
    We first derive an exploitable expression of $\operatorname{b} (\lambda)$. We have
    \begin{multline} \label{eqNonZero2}
        \operatorname{m}_+ (\lambda, x)
        = e^{-i 2 \lambda x} \left( \sum_{j \in \Z} \frac{\alpha_j e^{i 2 \lambda j}}{2 i \lambda} \operatorname{m}_+ (\lambda, j) \right) + \left( 1 - \sum_{j \in \Z} \frac{\alpha_j}{2 i \lambda} \operatorname{m}_+ (\lambda, j) \right) \\
        + \left( \sum_{j \leq x} \frac{1 - e^{2 i \lambda (j - x)}}{2 i \lambda} \alpha_j \operatorname{m}_+ (\lambda, j) \right)
    \end{multline}
    where all of the above sums are well-defined by \eqref{eqComput1} (if $\lambda \in \R^*$) and Lemma \eqref{lemBound}-\eqref{eqBound1} together with the fact that $\alpha \in \ell^{1, 1} (\Z)$ (if $\lambda \in \C^+$). 
    Furthermore,
    \begin{equation*}
        \sum_{j \leq x} \frac{1 - e^{2 i \lambda (j - x)}}{2 i \lambda} \alpha_j \operatorname{m}_+ (\lambda, j)
        = o_{x \to -\infty} (1)
    \end{equation*}
    using Lemma \ref{lemBound}-\eqref{eqBound1} in $s = j \in \Z$ together with $\alpha \in \ell^{1, 1} (\Z)$.
    Furthermore, by \eqref{eqCoeff2}, we have
    \begin{multline} \label{eqNonZero4}
        \operatorname{m}_+ (\lambda, x)
        = e^{-2 i \lambda x} \left( \sum_{j \in \Z} \left( \operatorname{a}_+ (\lambda) B_j^- (\lambda) + \operatorname{b} (\lambda) A_j^- (\lambda) \right) \boldsymbol{1}_{(j, j+1)} (x) \right) \\
        + \operatorname{b} (\lambda) \left( \sum_{j \in \Z} B_j^- (\lambda) \boldsymbol{1}_{(j, j+1)} (x) \right)
        + \operatorname{a}_+ (\lambda) \left( \sum_{j \in \Z} A_j^- (\lambda) \boldsymbol{1}_{(j, j+1)} (x) \right) \\
        = e^{-2 i \lambda x} \operatorname{a}_+ (\lambda) + \operatorname{b} (\lambda) + o_{x \to -\infty} (1)
    \end{multline}
    where we used $A_j^- (\lambda) \to 0$, $B_j^- (\lambda) \to 1$, $A_j^- (-\lambda) \to 0$ and $B_j^- (-\lambda) \to 1$ as $j \to -\infty$.
    Comparing \eqref{eqNonZero2} and \eqref{eqNonZero4} at $x \to -\infty$, we obtain that
    \begin{equation*}
        \operatorname{b} (\lambda)
        = 1 - \frac{1}{2 i \lambda} \sum_{j \in \Z} \alpha_j \operatorname{m}_+ (\lambda, j),
    \end{equation*}
    in particular $\lambda \in \C^+ \cup \R^* \mapsto \operatorname{b} (\lambda)$ is continuous. We now divide the proof into two cases.
    
    \emph{Step $1$.} Assume that $\sum_{j \in \Z} \alpha_j \operatorname{m}_+ (0, j) = 0$. Then,
    \begin{equation*}
        \operatorname{b} (\lambda)
        = 1 - \sum_{j \in \Z} \alpha_j \frac{\operatorname{m}_+ (\lambda, j) - \operatorname{m}_+ (0, j)}{2 i \lambda}.
    \end{equation*}
    However,
    \begin{equation*}
        \left| \frac{\operatorname{m}_+ (\lambda, j) - \operatorname{m}_+ (0, j)}{2 i \lambda} \right|
        \leq \frac{1}{2} \sup_{\lambda \in \C^+} \left| \partial_\lambda \operatorname{m}_+ (\lambda, j) \right|
        \lesssim 1 + j^2
    \end{equation*}
    by Lemma \ref{lemBound}-\eqref{eqBound2}.
    As $\alpha \in \ell^{1, 2} (\Z)$, we obtain that $\lim_{\lambda \to 0} \operatorname{b} (\lambda)$ exists and is finite. The fact that $|\operatorname{b} (\lambda)| \geq 1$ ensured by Proposition \ref{prpWronski} yields that $\lim_{\lambda \to 0} \operatorname{b} (\lambda) \neq 0$. We conclude that $\lambda \in \R \mapsto 1 / \operatorname{b} (\lambda)$ is continuous and never vanishes.
    
    \emph{Step $2$.} Assume that $\sum_{j \in \Z} \alpha_j \operatorname{m}_+ (0, j) \neq 0$. Then, by Lemma \ref{lemBound}-\eqref{eqBound2}, there exists a family $(\varepsilon_j)$ of functions in $\lambda$ such that, for each $j \in \Z$,
    \begin{equation*}
        |\varepsilon_j (\lambda)|
        \lesssim \left( 1 + j^2 \right) o_{\lambda \to 0} (\lambda)
    \end{equation*}
    and
    \begin{equation*}
        \operatorname{m}_+ (\lambda, j)
        = \operatorname{m}_+ (0, j) + \lambda \partial_\lambda \operatorname{m}_+ (0, j) + \varepsilon_j (\lambda).
    \end{equation*}
    Thus,
    \begin{align*}
        \operatorname{b} (\lambda)
        & = 1 - \frac{1}{2 i \lambda} \sum_{j \in \Z} \alpha_j \operatorname{m}_+ (0, j) - \frac{1}{2 i} \sum_{j \in \Z} \alpha_j \frac{\operatorname{m}_+ (\lambda, j) - \operatorname{m}_- (0, j)}{\lambda} \\
        & = 1 - \frac{1}{2 i \lambda} \sum_{j \in \Z} \alpha_j \operatorname{m}_+ (0, j) - \frac{1}{2 i} \sum_{j \in \Z} \alpha_j \left( \partial_\lambda \operatorname{m}_+ (0, j) + \varepsilon_j (\lambda) \right) \\
        & = \left( 1 - \frac{1}{2 i} \sum_{j \in \Z} \alpha_j \partial_\lambda \operatorname{m}_+ (0, j) \right) - \frac{1}{2 i \lambda} \sum_{j \in \Z} \alpha_j \operatorname{m}_+ (0, j) + o_{\lambda \to 0} (\lambda),
    \end{align*}
    all of the above sums being well defined by Lemma \ref{lemBound}.
    We thus get
    \begin{equation*}
        \frac{1}{\operatorname{b} (\lambda)}
        = C \lambda + o_{\lambda \to 0} (\lambda)
    \end{equation*}
    for some $C \neq 0$. Thus $\lambda \mapsto 1/b(\lambda)$ can be continuously prolonged in $0$ and does not vanish on $\C^+ \cup \R$. This concludes the proof.
\end{proof}


\subsection{Kernel of the resolvent}

This subsection is devoted to the proof of the following theorem, which provides an explicit characterization of the kernel of the operator $\operatorname{R}_\alpha (\lambda^2 \pm i 0)$ in terms of Jost solutions and derives a Stone formula.

\begin{theorem} \label{thResKer}
    Assume that there exists $\mu > 0$ such that $\alpha \in \ell^{1, 1 + \mu} (\Z)$. Let $\lambda > 0$. Then, the kernel $\operatorname{G}_\alpha (\lambda^2 \pm i 0): \R^2 \to \C$ of $\operatorname{R}_\alpha (\lambda^2 \pm i 0)$ is given by
    \begin{align} \label{eqResKer1} 
        \operatorname{G}_\alpha (\lambda^2 \pm i 0) (x, y) & \coloneqq
        \frac{1}{\operatorname{W} (\pm \lambda)}
        \begin{cases}
            \operatorname{f}_- (\pm \lambda, x) \operatorname{f}_+ (\pm \lambda, y) \quad \text{if } y > x, \\
            \operatorname{f}_- (\pm \lambda, y) \operatorname{f}_+ (\pm \lambda, x) \quad \text{if } y < x.
        \end{cases}
    \end{align}
    Furthermore, assume that $\operatorname{W} (0) \neq 0$ or $\alpha \in \ell^{1, 2} (\Z)$. Then, for any $\Phi : \R \to \C$ continuous and bounded, the Stone formula
    \begin{equation} \label{eqResKer7}
        \Phi (\H_\alpha) \operatorname{P}
        = \frac{1}{\pi i} \int_{-\infty}^\infty \lambda \Phi (\lambda^2) \operatorname{R}_\alpha (\lambda^2 + i 0) \, d \lambda
    \end{equation}
    holds.
\end{theorem}

\begin{proof}
    As $\alpha \in \ell^{1, 1 + \mu} (\Z)$, the operator $\operatorname{R}_\alpha (\lambda^2 \pm i 0)$ is well-defined by Theorem \ref{thLAP}. We first compute the kernel of $\operatorname{R}_\alpha (\lambda^2 + i 0)$. We proceed by analysis–synthesis approach and start by the analysis part of the argument.

    We consider the equation
    \begin{equation} \label{eqResKer4}
        (\H_{\alpha} - \lambda^2) G (x, y)
        = \delta (x - y)
    \end{equation}
    posed in the distributional sense. Let $y \in \R \setminus \Z$ and $\varepsilon > 0$ such that $[y - \varepsilon, y + \varepsilon] \cap \Z = \emptyset$. Then, for $x \in (y - \varepsilon, y + \varepsilon)$, the equation \eqref{eqResKer4} can be written as
    \begin{equation*}
        -\partial_{xx} G (x, y) - \lambda^2 G (x, y) = \delta (x - y).
    \end{equation*}
    Thus, $G (\cdot, y)$ satisfies the continuity condition
    \begin{equation} \label{eqResKer5}
        G (y+, y) = G (y-, y).
    \end{equation}
    Furthermore, by integrating \eqref{eqResKer4} on $(y - \varepsilon, y + \varepsilon)$ on the variable $x$ in the sense of distributions, we obtain
    \begin{equation*}
        -\partial_x G (y + \varepsilon, y) + \partial_x G (y - \varepsilon, y) - \lambda^2 \int_{y - \varepsilon}^{y + \varepsilon} G (x, y) \,  \, dx
        = 1.
    \end{equation*}
    Making $\varepsilon \to 0$, we obtain that $G (\cdot, y)$ satisfies the jump derivative condition
    \begin{equation} \label{eqResKer6}
        \partial_x G (x, y) \big|_{x = y+} - \partial_x G (x, y) \big|_{x = y-} = 1.
    \end{equation}
    Furthermore, by Theorem \ref{thLAP}, it satisfies
    the outgoing condition
    \begin{align}
        & (\partial_x - i \lambda) G (x, y) \to 0 \text{ as } x \to \infty, \label{eqResKer3a} \\
        & (\partial_x + i \lambda) G (x, y) \to 0 \text{ as } x \to -\infty, \label{eqResKer3b}
    \end{align}
    as $x \to -\infty$. The Jost solutions $\operatorname{f}_+ (\lambda, \cdot)$ and $\operatorname{f}_- (\lambda, \cdot)$ are solutions of the equation
    \begin{equation*}
        (\H_{\alpha} - \lambda^2) f
        = 0
    \end{equation*}
    posed in the distributional sense, which are linearly independent as $\operatorname{W} (\lambda) \neq 0$ if $\lambda > 0$ by Proposition \ref{prpWronski}. Furthermore, they satisfy
    \begin{align*}
        (\partial_x \mp i \lambda) \operatorname{f}_+ (\lambda, x) 
        & \to 0, \\
        (\partial_x \pm i \lambda) \operatorname{f}_- (\lambda, x)
        & \to 0,
    \end{align*}
    as $x \to \pm \infty$. Thus, there exists $A (y), B (y) \in \C$ such that
    \begin{equation*}
        G (x, y) =
        \begin{cases}
            A (y) \operatorname{f}_+ (\lambda, x) \text{ for } x > y, \\
            B (y) \operatorname{f}_- (\lambda, x) \text{ for } x < y.
        \end{cases}
    \end{equation*}
    In order for $G$ to satisfy \eqref{eqResKer5}-\eqref{eqResKer6}, for all $y \in \R \setminus \Z$, the coefficients $A$ and $B$ must satisfy
    \begin{align*}
        A (y) \operatorname{f}_+ (\lambda, y) - B (y) \operatorname{f}_- (\lambda, y) & = 0, \\
        A (y) \partial_x \operatorname{f}_+ (\lambda, x) \big|_{x = y} - B (y) \partial_x \operatorname{f}_- (\lambda, x) \big|_{x = y} & = 1,
    \end{align*}
    so that
    \begin{equation*}
        A (y) = \frac{\operatorname{f}_- (\lambda, y)}{\operatorname{W} (\lambda)}, \quad B (y) = \frac{\operatorname{f}_+ (\lambda, y)}{\operatorname{W} (\lambda)}.
    \end{equation*}
    This concludes the analysis part of the argument.
    
    Using the explicit expression of the Jost solution given by \eqref{eqJostSol14}-\eqref{eqJostSol15}, we obtain that $G (\lambda^2 + i 0)$, given by \eqref{eqResKer1}, is a solution of the equation \eqref{eqResKer4} and satisfies \eqref{eqResKer3a}-\eqref{eqResKer3b}. This concludes the synthesis part of the argument. We compute similarly the kernel of $\operatorname{R}_\alpha (\lambda^2 - i 0)$. This completes the proof of \eqref{eqResKer1}.
    

    Finally, let $\Phi : \R \to \C$ be continuous and bounded. Then, since $\sigma_{ess} (\H_\alpha) = [0,\infty)$, the Stone formula
    \begin{align*} 
        \Phi (\H_\alpha) \operatorname{P}
        & = \frac{1}{2 \pi i} \int_{0}^\infty \Phi (\lambda) \left( \operatorname{R}_\alpha (\lambda + i 0) - \operatorname{R}_\alpha (\lambda - i 0) \right) \, d \lambda \\
        & = \frac{1}{\pi i} \int_{0}^\infty \lambda \Phi (\lambda^2) \left( \operatorname{R}_\alpha (\lambda^2 + i 0) - \operatorname{R}_\alpha (\lambda^2 - i 0) \right) \, d \lambda,
    \end{align*}
    holds (see \cite[Theorem VII.13]{ReSi81}), where we made the change of variable $\lambda \mapsto \lambda^2$ in order to pass to the second line. Thus, we have
    \begin{align*} 
        \pi i \Phi (\H_\alpha) \operatorname{P}
        & = \int_0^\infty \lambda \Phi (\lambda^2) \operatorname{R}_\alpha (\lambda^2 + i 0) \, d \lambda + \int_0^\infty (-\lambda) \Phi ((-\lambda)^2) \operatorname{R}_\alpha ((-\lambda)^2 + i 0) \, d \lambda \\
        & = \int_{-\infty}^\infty \lambda \Phi (\lambda^2) \operatorname{R}_\alpha (\lambda^2 + i 0) \, d \lambda,
    \end{align*}
    where the first equality holds as the kernel of $\operatorname{R}_\alpha$ satisfies
    \begin{equation*}
        \operatorname{G}_\alpha (\lambda^2 - i 0)
        = \operatorname{G}_\alpha ((-\lambda)^2 + i 0).
    \end{equation*}
    Then, \eqref{eqResKer7} follows and this concludes the proof.
\end{proof}

\section{Dispersive estimate} \label{sec5}

In this section, we prove Theorem \ref{thMain} by splitting the Stone formula into two regimes: for values of $\lambda > \| \alpha \|_{\ell^1 (\Z)}$, also called the \textit{high-energy} part; and for values of $\lambda < \| \alpha \|_{\ell^1 (\Z)}$, also called the \textit{low-energy} part. The approach used was developed in \cite[Section $2$]{GoSc04}.

\subsection{High-energy part}

This subsection is dedicated to prove high-energy part of the estimate, stated in the following proposition.

\begin{proposition} \label{prpHighEn}
    Assume that there exists $\mu > 0$ such that $\alpha \in \ell^{1, 1 + \mu} (\Z)$. Let $\chi$ be a smooth cut-off such that $\chi (\lambda) = 0$ for $|\lambda| \leq \| \alpha \|_{\ell^1 (\Z)}^2$ and $\chi (\lambda) = 1$ for $|\lambda| \geq 2 \| \alpha \|_{\ell^1 (\Z)}^2$. Then, for all $t \in \R$ and $f \in \mathcal{S} (\R)$, the following estimate
    \begin{equation} \label{eqHighEn1}
        \left\| e^{i t \H_\alpha} \chi (\H_\alpha) \operatorname{P} f \right\|_{L^\infty (\R)}
        \lesssim |t|^{-\frac{1}{2}} \| f \|_{L^1 (\R)}
    \end{equation}
    holds.
\end{proposition}

The proof relies on the Born series expansion of $\operatorname{R}_\alpha (\lambda^2 + i 0)$. As stated in Section \ref{sec3}, since the operators $\H_{\alpha}$ and $\H_0$ are not defined on the same domain, the resolvent formula cannot be applied directly and, instead, we work with their Friedrichs extensions $\tilde{\H}_\alpha$ and $\tilde{\H}_0$, which share a common domain. The following lemma holds.

\begin{lemma} \label{lemBorn}
    Assume that there exists $\mu > 0$ such that $\alpha \in \ell^{1, 1 + \mu} (\Z)$. Let $|\lambda| > \| \alpha \|_{\ell^1 (\Z)}$ and $f, g \in \mathcal{S} (\R)$. Then, the Born series expansion
    \begin{equation} \label{eqBorn1}
        \left\langle \operatorname{R}_\alpha (\lambda^2 + i 0) f, g \right\rangle_{L^\infty (\R), L^1 (\R)}
        = \sum_{n = 0}^\infty \left\langle \tilde{\operatorname{R}}_{0} (\lambda^2 + i 0) \left( (\tilde{\H}_\alpha - \tilde{\H}_0) \tilde{\operatorname{R}}_{0} (\lambda^2 + i 0) \right)^n f, g \right\rangle_{L^\infty (\R), L^1 (\R)}
    \end{equation}
    holds.
\end{lemma}

\begin{proof}
    Let $\lambda \in \R^*$ and $f \in \mathcal{S} (\R)$,  For any $s \in (1/2, (1 + \mu)/2)$, the inclusion $\mathcal S (\R) \subset L^{2, s} (\R)$ holds. Furthermore, since $\lambda \neq 0$, Proposition \ref{prpJostSol} ensures that $\operatorname{f}_+ (\pm \lambda, \cdot), \operatorname{f}_- (\pm \lambda, \cdot) \in L^\infty (\R)$. Thus, by Theorem \ref{thResKer}, we have $\operatorname{R}_\alpha (\lambda^2 + i 0) f \in L^\infty (\R)$.
    
    We first compute
    \begin{equation*}
        (\tilde{\H}_0 - \tilde{\H}_\alpha) \tilde{\operatorname{R}}_0 (\lambda^2 + i 0) f
        = - \frac{i}{2 \lambda} \sum_{j \in \Z}  \left( \alpha_j \int_{-\infty}^\infty e^{i \lambda |j - y|} f (y) \, dy \right) \delta(\cdot - j).
    \end{equation*}
    For $n \in \N$, we obtain by induction that
    \begin{multline*}
        \left( (\tilde{\H}_0 - \tilde{\H}_\alpha) \tilde{\operatorname{R}}_0 (\lambda^2 + i 0) \right)^n f
        = \left( -\frac{i}{2 \lambda} \right)^n \sum_{j_1, \ldots, j_n \in \Z} \left( \alpha_{j_1} \int_{-\infty}^\infty e^{i \lambda |j_1 - y|} f (y) \, dy \right) \\
        \times \left( \prod_{k = 2}^n \alpha_{j_k} e^{i \lambda |j_{k-1} - j_k|} \right) \delta(\cdot - j_n).
    \end{multline*}
    Thus, we have
    \begin{multline*}
        \tilde{\operatorname{R}}_0 (\lambda^2 + i 0) \left( (\tilde{\H}_0 - \tilde{\H}_\alpha) \tilde{\operatorname{R}}_0 (\lambda^2 + i 0) \right)^{n} f
        = \left( -\frac{i}{2 \lambda} \right)^{n+1} \sum_{j_1, \ldots, j_n \in \Z} \left( \alpha_{j_1} \int_{-\infty}^\infty e^{i \lambda |j_1 - y|} f (y) \, dy \right) \\
        \times \left( \prod_{k = 2}^n \alpha_{j_k} e^{i \lambda |j_{k-1} - j_k|} \right) e^{i \lambda |\cdot - j_n|},
    \end{multline*}
    so that
    \begin{multline} \label{eqBorn3}
        \left\langle \tilde{\operatorname{R}}_0 (\lambda^2 + i 0) \left( (\tilde{\H}_0 - \tilde{\H}_\alpha) \tilde{\operatorname{R}}_0 (\lambda^2 + i 0) \right)^{n} f, g \right\rangle_{L^\infty (\R), L^1 (\R)} \\
        = \left( -\frac{i}{2 \lambda} \right)^{n+1} \int_{-\infty}^\infty \sum_{j_1, \ldots, j_n \in \Z} \left( \alpha_{j_1} \int_{-\infty}^\infty e^{i \lambda |j_1 - y|} f (y) \, dy \right) \left( \prod_{k = 2}^n \alpha_{j_k} e^{i \lambda |j_{k-1} - j_k|} \right) e^{i \lambda |x - j_n|} \, \overline{g (x)} \,  \, dx.
    \end{multline}
    Therefore,
    \begin{multline} \label{eqBorn2}
        \left| \left\langle \tilde{\operatorname{R}}_0 (\lambda^2 + i 0) \left( (\tilde{\H}_0 - \tilde{\H}_\alpha) \tilde{\operatorname{R}}_0 (\lambda^2 + i 0) \right)^{n} f, g \right\rangle_{L^\infty (\R), L^1 (\R)} \right| \\
        \leq \frac{1}{2 \lambda} \left( \frac{\| \alpha \|_{\ell^1 (\Z)}}{2 \lambda} \right)^{n} \| f \|_{L^1 (\R)} \| g \|_{L^1 (\R)},
    \end{multline}
    which is the term of a converging serie in $n$ as $\lambda > \| \alpha \|_{\ell^1 (\Z)}$.
    
    As $\alpha \in \ell^{1, 1 + \mu} (\Z)$, we have that the operator $\operatorname{R}_\alpha (\lambda^2 \pm i 0)$ is well-defined by Theorem \ref{thLAP}. Iterating the resolvent identity \eqref{eqResExt1}, we obtain
    \begin{align*}
        \left\langle \operatorname{R}_\alpha (\lambda^2 + i 0) f, g \right\rangle_{L^\infty (\R), L^1 (\R)}
        & = \left\langle \tilde{\operatorname{R}}_\alpha (\lambda^2 + i 0) f, g \right\rangle_{L^\infty (\R), L^1 (\R)} \\
        & = \sum_{n = 0}^\infty \left\langle \tilde{\operatorname{R}}_0 (\lambda + i 0) \left( (\tilde{\H}_0 - \tilde{\H}_\alpha) \tilde{\operatorname{R}}_0 (\lambda + i 0) \right)^n f, g \right\rangle_{L^\infty (\R), L^1 (\R)},
    \end{align*}
    where the convergence of the serie is justified by \eqref{eqBorn2}. This concludes the proof.
\end{proof}

We turn to the proof of the high–energy part of the estimate.

\begin{proof}[Proof of Proposition \ref{prpHighEn}]
    Let $\phi \in C_0^\infty (\R)$ be such that $\phi (\lambda) = 1$ if $|\lambda| < 1$ and $\phi (\lambda) = 0$ if $|\lambda| > 2$. Let $L > 0$, $\chi_L$ be the truncated cut-off given by $\chi_L = \chi (\cdot) \phi (\cdot / L)$ and $g \in \mathcal{S} (\R)$. We have
    \begin{multline*}
        \left\langle e^{i t \H_\alpha} \chi_L (\H_\alpha) f, g \right\rangle_{L^\infty (\R), L^1 (\R)}
        = \frac{1}{\pi i} \sum_{n = 0}^\infty \int_{-\infty}^\infty \lambda e^{i t \lambda^2} \chi_L (\lambda^2) \\
        \times \left\langle \tilde{\operatorname{R}}_{0} (\lambda^2 + i 0) \left( (\tilde{\H}_\alpha - \tilde{\H}_0) \tilde{\operatorname{R}}_{0} (\lambda^2 + i 0) \right)^n f, g \right\rangle_{L^\infty (\R), L^1 (\R)} \, d \lambda
    \end{multline*}
    where we used Stone formula \eqref{eqResKer7} and Born series expansion \eqref{eqBorn1}. Thus,
    \begin{multline*}
        \left| \left\langle e^{i t \H_\alpha} \chi_L (\H_\alpha) f, g \right\rangle_{L^\infty (\R), L^1 (\R)} \right|
        \lesssim \sum_{n = 0}^\infty 2^{-(n + 1)} \sum_{J = (j_1, \dots, j_n) \in \Z^n} |\alpha_{j_1} \dots \alpha_{j_n}| \int_{\R^2} |f (y)| |g (x)| \\
        \times \left| \int_{-\infty}^\infty e^{i t \lambda^2} e^{i \lambda (|j_1 - y| + |x - j_n| + \sum_{k=2}^n |j_k - j_{k-1}|)} \frac{\chi_L (\lambda^2)}{ \lambda^{n}} \, d \lambda \right| \, dy \,  \, dx,
    \end{multline*}
    by \eqref{eqBorn3} and Fubini's theorem. We obtain
    \begin{multline*}
        \left| \left\langle e^{i t \H_\alpha} \chi_L (\H_\alpha) f, g \right\rangle_{L^\infty (\R), L^1 (\R)} \right| \\
        \lesssim \sum_{n = 0}^\infty 2^{-(n + 1)} \| f \|_{L^1 (\R)} \| g \|_{L^1 (\R)} \sup_{a \in \R} \left| \int_{-\infty}^\infty e^{i (t \lambda^2 + a \lambda)} \chi_L (\lambda^2) \left( \frac{\| \alpha \|_{\ell^1 (\Z)}}{\lambda} \right)^n \, d \lambda \right| \\
        \lesssim \sum_{n = 0}^\infty 2^{-(n + 1)} \| f \|_{L^1 (\R)} \| g \|_{L^1 (\R)} \| e^{i t \H_0} \psi \|_{L^\infty (\R)},
    \end{multline*}
    where $\psi = \mathcal{F} (\lambda \in \R \mapsto \chi_L (\lambda^2) \| \alpha \|_{\ell^1 (\Z)}^n \lambda^{-n})$. Moreover, we have that $\psi \in L^1 (\R)$ (see \cite[equation $(11)$]{GoSc04}). By \eqref{eqFLDis}, we obtain
    \begin{equation*}
        \sup_{L \geq 1} \left| \left\langle e^{i t \H_\alpha} \chi_L (\H_\alpha) f, g \right\rangle_{L^\infty (\R), L^1 (\R)} \right|
        \lesssim |t|^{-\frac{1}{2}} \| f \|_{L^1 (\R)} \| g \|_{L^1 (\R)}.
    \end{equation*}
    This concludes the proof.
\end{proof}

\subsection{Low-energy part}

This subsection is dedicated to prove the low-energy part of the estimate, stated in the following proposition.

\begin{proposition} \label{prpLowEn}
    Assume that there exists $\mu \in (0, 1)$ such that $\alpha \in \ell^{1, 1 + \mu} (\Z)$ and $\operatorname{W} (0) \neq 0$; or that $\operatorname{W} (0) = 0$ and $\alpha \in \ell^{1, 2} (\Z)$. Let $\chi$ be a smooth and compactly supported cut-off. Then, for all $t \in \R^*$ and $f \in \mathcal S (\R)$, the following estimate
    \begin{equation} \label{eqLowEn1}
        \left\| e^{i t \H_\alpha} \chi (\H_\alpha) \operatorname{P} f \right\|_{L^\infty (\R)}
        \lesssim |t|^{-\frac{1}{2}} \| f \|_{L^1 (\R)}
    \end{equation}
    holds.
\end{proposition}

The proof relies on the properties of the Jost solutions. Let $\lambda \in \C^+ \cup \R$ and $x \in \R$. We consider $\operatorname{m}_\pm (\lambda, x)$ introduced in \eqref{eqm+-}. The following lemma states that the function $\operatorname{m}_\pm (\cdot, x) - 1$ belongs to the Hardy space on $\C^+$.

\begin{lemma} \label{lemHardy}
    Assume that $\alpha \in \ell^1 (\Z)$. Let $x \in \R$. Then, $\operatorname{m}_\pm (\cdot, x) - 1 \in \mathcal H (\C^+)$.
\end{lemma}

\begin{proof}
    By Proposition \ref{prpJostSol} and the expression of $\operatorname{m}_\pm (\cdot, x)$ given by \eqref{eqM+1}, we obtain that $\operatorname{m}_+ (\cdot, x) - 1$ is holomorphic in $\C^+$. Let $\lambda \in \C^+$, we have
    \begin{align*}
        \left| \operatorname{m}_+ (\lambda, x) - 1 \right|
        & \leq \sum_{n = 1}^\infty \left( \frac{1}{2 |\lambda|} \right)^n \sum_{x < j_1 < \ldots < j_n} \prod_{l = 1}^n 2 \left| \alpha_{j_l} \right| \\
        & \leq \sum_{n = 1}^\infty \frac{1}{n!} \left( \frac{\| \alpha \|_{\ell^1 (\Z)}}{|\lambda|} \right)^n \\
        & \leq \operatorname{exp} \left( \frac{\| \alpha \|_{\ell^1 (\Z)}}{|\lambda|} \right) - 1,
    \end{align*}
    where we used \eqref{eqm+3}; $|e^{2 i \lambda (j_1 -x)}| \leq 1$ and $|e^{2 i \lambda (j_l - j_{l-1})}| \leq 1$, as $\operatorname{Im} (\lambda) > 0$; and $j_l - j_{l-1} > 0$. Hence,
    \begin{equation*}
        \left| \operatorname{m}_+ (\lambda, x) - 1 \right|
        = O \left( |\lambda|^{-1} \right) \text{ as } |\lambda|
        \to \infty \text{ in } \C^+.
    \end{equation*}
    This suffices to conclude that $\operatorname{m}_+ (\cdot, x) - 1 \in \mathcal H (\C^+)$. We prove similarly that $\operatorname{m}_- (\cdot, x) - 1 \in \mathcal H (\C^+)$, hence the proof.
\end{proof}

The Fourier transform of $\lambda \in \R \mapsto \operatorname{m}_\pm (\lambda, x) - 1$, denoted by $\xi \in \R \mapsto \mathcal{F} (\operatorname{m}_\pm (\lambda, x) - 1) (\xi)$, is a complex measure in the distribution sense. By Lemma \ref{lemHardy} and the Hardy space properties, it is supported in $\R^+$. By Paley-Wiener theory, we obtain that $\mathcal{F} (\operatorname{m}_\pm (\lambda, x) - 1) \in L^2 (\R^+)$ such that
\begin{equation}
    \operatorname{m}_\pm (\lambda, x) - 1 = \int_0^\infty \mathcal{F} (\operatorname{m}_\pm (\lambda, x) - 1) (\xi) e^{2 i \lambda \xi} \, d \xi,
\end{equation}
where we recall that, by abuse of notation, we denote by $\xi \mapsto \mathcal{F} (g (\lambda, x)) (\xi)$ the Fourier transform in $\lambda$ of the parametrized function $g (\cdot, x): \lambda \mapsto g (\lambda, x)$.

The following lemma shows that the total variation norm of $\mathcal{F} (\operatorname{m}_\pm (\lambda, x) - 1)$ is in fact bounded. It was proved in \cite[Lemma $3$]{DeTr79}, in the case of a potential $V$ which is a function.

\begin{lemma} \label{lemFourier}
    Let $j \in \{ 1, 2 \}$ and assume that $\alpha \in \ell^{1, j} (\Z)$. Let $x \in \R$, then
    \begin{align}
        & \sup_{\pm x \geq 0} \left\| \xi^{j-1} \mathcal{F} (\operatorname{m}_\pm (\lambda, x) - 1) (\xi) \right\|_{\mathcal{M}} < \infty, \label{eqFourier1} \\
        & \sup_{\pm x \geq 0} \left\| \xi^{j-1} \partial_x \mathcal{F} (\operatorname{m}_\pm (\lambda, x) - 1) (\xi) \right\|_{\mathcal M} < \infty, \label{eqFourier2} \\
        & \sup_{\pm x \geq 0} \left\| \xi^{j-1} \partial_\xi \mathcal{F} (\operatorname{m}_\pm (\lambda, x) - 1) (\xi) \right\|_{\mathcal M} < \infty. \label{eqFourier3}
    \end{align}
    Furthermore, let $\tilde{\chi}$ be a smooth and compactly supported cut-off; and $\mathcal{F} (\tilde{\chi})$ be its Fourier transform. Then,
    \begin{align}
        & \sup_{\pm x \geq 0} \left\| \xi^{j-1} \left( \mathcal{F} (\tilde{\chi}) \ast \mathcal{F} (\operatorname{m}_\pm (\lambda, x) - 1) \right) (\xi) \right\|_{\mathcal M} < \infty, \\
        & \sup_{\pm x \geq 0} \left\| \xi^{j-1} \partial_x \left( \mathcal{F} (\tilde{\chi}) \ast \mathcal{F} (\operatorname{m}_\pm (\lambda, x) - 1) \right) (\xi) \right\|_{\mathcal M} < \infty, \\
        & \sup_{\pm x \geq 0} \left\| \xi^{j-1} \partial_\xi \left( \mathcal{F} (\tilde{\chi}) \ast \mathcal{F} (\operatorname{m}_\pm (\lambda, x) - 1) \right) (\xi) \right\|_{\mathcal M} < \infty.
    \end{align}
\end{lemma}

\begin{proof}
    Let $\xi \geq 0$. By computing the Fourier transform of $\operatorname{m}_+ (\cdot, x) - 1$ using formula \eqref{eqm+3}, we obtain that
    \begin{equation*}
        \mathcal{F} (\operatorname{m}_+ (\lambda, x) - 1) (\xi)
        = \sum_{n = 1}^\infty \pi^n \sum_{x < j_1 < \ldots < j_n} \left( \prod_{l = 1}^n \alpha_{j_l} \right) P_{j_1, \ldots, j_n} (\xi, x)
    \end{equation*}
    where, for $\xi \in \R$,
    \begin{equation*}
        P_{j_1, \ldots, j_n} (\xi, x)
        = \left( \boldsymbol{1}_{[0, j_1 - x]} \ast \boldsymbol{1}_{[0, j_2 - j_1]} \ast \cdots \ast \boldsymbol{1}_{[0, j_n - j_{n-1}]} \right) (\xi).
    \end{equation*}
    In particular, $P_{j_1, \ldots, j_n} (\xi, x) = 0$ if $\xi < 0$. We divide the rest of the proof in two steps: the first treats the case $\alpha \in \ell^{1, 1} (\Z)$ while the other addresses the case $\alpha \in \ell^{1, 2} (\Z)$. In both steps, we assume $0 \leq x = j_0 < j_1 < \ldots < j_n$.

    \medskip
    
    \noindent \emph{Step $1$.} Assume that $\alpha \in \ell^{1, 1} (\Z)$. We have
    \begin{align*}
        \int_0^\infty P_{j_1, \ldots, j_n} (\xi, x) \, d \xi
        & = \int_0^\infty \int_{-\infty}^\infty \boldsymbol{1}_{[0, j_1 - x]} (\xi - y) P_{j_2, \ldots, j_n} (y, j_1) \, dy \, d \xi \\
        & = \int_{-\infty}^\infty \int_0^\infty \boldsymbol{1}_{[0, j_1 - x]} (\xi - y) \, d \xi P_{j_2, \ldots, j_n} (y, j_1) \, dy \\
        & = \left( j_1 - x \right) \int_{-\infty}^\infty P_{j_2, \ldots, j_n} (y, j_1) \, dy,
    \end{align*}
    where we used Fubini's theorem. Thus, by induction, we obtain
    \begin{equation} \label{eqFourier9}
        \int_0^\infty |P_{j_1, \ldots, j_n} (\xi, x)| \, d \xi
        \leq \prod_{l = 1}^n \left( j_l - j_{l-1} \right)
        \leq \prod_{l = 1}^n j_l.
    \end{equation}
    Hence,
    \begin{equation} \label{eqFourier4}
        \begin{aligned}
            \sup_{x \geq 0} \| \mathcal{F} (\operatorname{m}_+ (\lambda, x) - 1) \|_{\mathcal M}
            & \leq \sum_{n = 1}^\infty \pi^n \sum_{x < j_1 < \ldots < j_n} \prod_{l = 1}^n |j_l \alpha_{j_l}| \\
            & \leq \sum_{n = 1}^\infty \frac{(\pi \| \alpha \|_{\ell^{1, 1} (\Z)})^n}{n!}
            \leq e^{\pi \| \alpha \|_{\ell^{1, 1} (\Z)}}
            < \infty.
        \end{aligned}
    \end{equation}
    Furthermore, we have
    \begin{multline} \label{eqFourier7}
        \partial_x \left( \sum_{x < j_1 < \ldots < j_n} \left( \prod_{l = 1}^n \alpha_l \right) P_{j_1, \ldots, j_n} (\xi, x) \right)
        = \sum_{x < j_1 < \ldots < j_n} \left( \prod_{l = 1}^n \alpha_l \right) \\
        \times \left( \partial_x P_{j_1, \ldots, j_n} (\xi, x) - P_{j_1, \ldots, j_n} (\xi, x) \delta (x - j_1) \right)
    \end{multline}
    and
    \begin{equation} \label{eqFourier8}
        \partial_\xi \left( \sum_{x < j_1 < \ldots < j_n} \left( \prod_{l = 1}^n \alpha_l \right) P_{j_1, \ldots, j_n} (\xi, x) \right)
        = \sum_{x < j_1 < \ldots < j_n} \left( \prod_{l = 1}^n \alpha_l \right) \partial_\xi P_{j_1, \ldots, j_n} (\xi, x),
    \end{equation}
    with
    \begin{gather*}
        \partial_x P_{j_1, \ldots, j_n} (\xi, x)
        = P_{j_2, \ldots, j_n} (\xi - (j_1 - x), j_1), \\
        \partial_\xi P_{j_1, \ldots, j_n} (\xi, x)
        = P_{j_2, \ldots, j_n} (\xi, j_1) - P_{j_2, \ldots, j_n} (j_1 - x - \xi, j_1).
    \end{gather*}
    Thus,
    \begin{gather*}
        \int_0^\infty |\partial_x P_{j_1, \ldots, j_n} (\xi, x)| \, d \xi + \| P_{j_1, \ldots, j_n} (\xi, \cdot) \delta (\cdot - j_1) \|_{\mathcal{M}}
        \leq 2 \prod_{l = 1}^n |j_l|, \\
        \int_0 ^\infty \left| \partial_\xi P_{j_1, \ldots, j_n} (\xi, x) \right| \, d \xi
        \leq 2 \prod_{l = 1}^n |j_l|.
    \end{gather*}
    Therefore,
    \begin{align}
        & \sup_{x \geq 0} \| \partial_x \mathcal{F} (\operatorname{m}_+ (\lambda, x) - 1) \|_{\mathcal{M}}
        \leq 2 \sum_{n = 1}^\infty \pi^n \sum_{x < j_1 < \ldots < j_n} \prod_{l = 1}^n |j_l \alpha_{j_l}|
        \leq 2 e^{\pi \| \alpha \|_{\ell^{1, 1} (\Z)}}
        < \infty, \label{eqFourier5} \\
        & \sup_{x \geq 0} \| \partial_\xi \mathcal{F} (\operatorname{m}_+ (\lambda, x) - 1) \|_{\mathcal{M}}
        \leq 2 \sum_{n = 1}^\infty \pi^n \sum_{x < j_1 < \ldots < j_n} \prod_{l = 1}^n |j_l \alpha_{j_l}|
        \leq 2 e^{\pi \| \alpha \|_{\ell^{1, 1} (\Z)}}
        < \infty. \label{eqFourier6}
    \end{align}
    We obtain identical bounds as \eqref{eqFourier4} (respectively \eqref{eqFourier5} and \eqref{eqFourier6}) for $\mathcal{F} (\operatorname{m}_- (\lambda, x) - 1)$ and $\mathcal{F} (\tilde{\chi}) \ast \mathcal{F} (\operatorname{m}_\pm (\lambda, x) - 1)$, proving the other estimates for $j = 1$. This concludes the first step of the proof.

    \medskip
    
    \noindent \emph{Step $2$.} Assume that $\alpha \in \ell^{1, 2} (\Z)$. We have
    \begin{align*}
        \int_0^\infty \xi P_{j_1, \ldots, j_n} (\xi, x) \, d \xi
        & = \int_0^\infty \xi \int_{-\infty}^\infty \boldsymbol{1}_{[0, j_1 - x]} (\xi - y) P_{j_2, \ldots, j_n} (y, j_1) \, dy \, d \xi \\
        & = \int_{-\infty}^\infty \int_0^\infty \xi \boldsymbol{1}_{[0, j_1 - x]} (\xi - y) \, d \xi P_{j_2, \ldots, j_n} (y, j_1) \, dy \\
        & = \int_{-\infty}^\infty \frac{1}{2} \left( (y + j_1 - x)^2 - y_+^2 \right)_+ P_{j_2, \ldots, j_n} (y, j_1) \, dy \\
        & = \frac{(j_1 - x)^2}{2} \int_{0}^\infty P_{j_2, \ldots, j_n} (y, j_1) \, dy + (j_1 - x) \int_{0}^\infty y P_{j_2, \ldots, j_n} (y, j_1) \, dy,
    \end{align*}
    where we used Fubini's theorem and the fact that $P_{j_2, \ldots, j_n} (\cdot, j_1)$ is supported in $\R_+$. Iterating the previous computation, we obtain
    \begin{multline*}
        2 \int_0^\infty \xi P_{j_1, \ldots, j_n} (\xi, x) \, d \xi
        = (j_1 - x)^2 \int_{0}^\infty P_{j_2, \ldots, j_n} (y, j_1) \, dy \\
        + (j_1 - x) (j_2 - x)^2 \int_{0}^\infty P_{j_3, \ldots, j_n} (y, j_2) \, dy \\
        + \ldots + (j_1 - x) (j_2 - x) \cdots (j_{n-2} - x) (j_{n-1} - x)^2 \int_{0}^\infty P_{j_n} (y, j_{n-1}) \, dy \\
        + 2 (j_1 - x) (j_2 - x) \cdots (j_{n-2} - x) (j_{n-1} - x) \int_0^\infty y P_{j_n} (y, j_{n-1}) \, dy.
    \end{multline*}
    Using \eqref{eqFourier9} and the fact that
    \begin{equation*}
        \int_0^\infty y P_{j_n} (y, j_{n-1}) \, dy
        = \frac{(j_n - j_{n-1})^2}{2},
    \end{equation*}
    we obtain
    \begin{align*}
        \int_0^\infty \left| \xi P_{j_1, \ldots, j_n} (\xi, x) \right| \, d \xi
        & \lesssim \sum_{p = 1}^n j_p - x \prod_{l = 1}^n (j_l - x) \\
        & \leq \sum_{p = 1}^n j_p \prod_{l = 1}^n j_l
        \leq n \prod_{l = 1}^n j_l^2.
    \end{align*}
    Hence,
    \begin{align*}
        \sup_{x \geq 0} \| \xi \mathcal{F} (\operatorname{m}_+ (\lambda, x) - 1) (\xi) \|_{\mathcal M}
        & \leq \sum_{n = 1}^\infty n \pi^n \sum_{x < j_1 < \ldots < j_n} \prod_{l = 1}^n |j_l^2 \alpha_{j_l}| \\
        & \leq \sum_{n = 1}^\infty \frac{(\pi \| \alpha \|_{\ell^{1, 2}})^n}{(n - 1)!} \\
        & \leq \pi \| \alpha \|_{\ell^{1, 2} (\Z)} e^{\pi \| \alpha \|_{\ell^{1, 2} (\Z)}}
        < \infty.
    \end{align*}
    Using the above bound and \eqref{eqFourier7} (respectively \eqref{eqFourier8}), we prove \eqref{eqFourier2} (respectively \eqref{eqFourier3}).
    The other estimates are proved similarly. This concludes the proof.
\end{proof}

We deduce the following lemma.

\begin{lemma} \label{lemFBound}
    Let $j \in \{ 1, 2 \}$ and assume that $\alpha \in \ell^{1, j} (\Z)$. Let $\tilde{\chi}$ be a smooth and compactly supported cut-off. Then, $\mathcal{F} (\tilde{\chi} \operatorname{W})$ and $\mathcal{F} ( W (\operatorname{f}_+ (\lambda, x), \operatorname{f}_- (-\lambda, x)))$ (where the Fourier transform is taken in $\lambda$) belongs to $L^{1, j-1} (\R)$.
\end{lemma}

\begin{proof}
    For any $\lambda \in \R$, we have
    \begin{equation*}
        \tilde{\chi} (\lambda) \operatorname{W} (\lambda)
        = \tilde{\chi} (\lambda) \left( \operatorname{m}_+ (\lambda, 0) \partial_x \operatorname{m}_- (\lambda, 0) - \partial_x \operatorname{m}_+ (\lambda, 0) \operatorname{m}_- (\lambda, 0) \right) - 2 i \lambda \chi (\lambda) \operatorname{m}_+ (\lambda, 0) \operatorname{m}_- (\lambda, 0)
    \end{equation*}
    and
    \begin{equation*}
        W \left( \operatorname{f}_+ (\lambda, x), \operatorname{f}_- (-\lambda, x) \right)
        = \operatorname{m}_+ (\lambda, 0) \partial_x \operatorname{m}_- (\lambda, 0) - \partial_x \operatorname{m}_+ (\lambda, 0) \operatorname{m}_- (-\lambda, 0).
    \end{equation*}
    Thus, the Fourier transform of those functions will be a convolution product of functions belonging to $L_{(j-1)}^1 (\R)$ by Lemma \ref{lemFourier}. This concludes the proof.
\end{proof}

The proof of Proposition \ref{prpLowEn} relies on the Wiener's lemma stated below. See e.g. \cite[Chapter VIII, Lemma $3$]{Ka04} for its proof.

\begin{lemma}[Wiener's lemma] \label{lemWiener}
    Let $f: \R \to \C$ be a function such that $\mathcal{F} (f) \in L^1 (\R)$ and let $\chi \in C_c^\infty (\R)$ be a cut-off. Assume that $f (x) \neq 0$ for all $x \in \operatorname{supp} (\chi)$. Then, $\mathcal{F} \left( (\chi f)^{-1} \right) \in L^1 (\R)$.
\end{lemma}

We turn to the proof of the low–energy part of the estimate.

\begin{proof}[Proof of Proposition \ref{prpLowEn}]
    Let $x \in \R$. By Theorem \ref{thResKer}, we have
    \begin{multline*}
        e^{i t \H_\alpha} \chi (\H_\alpha) f (x)
        = \frac{1}{\pi i} \int_{-\infty}^\infty \frac{\lambda e^{i t \lambda^2} \chi (\lambda^2)}{\operatorname{W} (\lambda)} \\
        \times \left( \int_{-\infty}^x \operatorname{f}_- (\lambda, y) \operatorname{f}_+ (\lambda, x) f (y) \, dy + \int_x^\infty \operatorname{f}_- (\lambda, x) \operatorname{f}_+ (\lambda, y) f (y) \, dy \right) \, d \lambda.
    \end{multline*}
    Thus, by Fubini's theorem, we obtain
    \begin{multline} \label{eqLowEn4}
        \left| e^{i t \H_\alpha} \chi (\H_\alpha) f (x) \right|
        \lesssim \int_{-\infty}^\infty |f (y)| \left| \int_{-\infty}^\infty \frac{\lambda e^{i t \lambda^2} \chi (\lambda^2)}{\operatorname{W} (\lambda)} \operatorname{f}_- (\lambda, y) \operatorname{f}_+ (\lambda, x) \boldsymbol{1}_{y < x} \, d \lambda \right| \, dy \\
        + \int_{-\infty}^\infty |f (y)| \left| \int_{-\infty}^\infty \frac{\lambda e^{i t \lambda^2} \chi (\lambda^2)}{\operatorname{W} (\lambda)} \operatorname{f}_- (\lambda, x) \operatorname{f}_+ (\lambda, y) \boldsymbol{1}_{y > x} \, d \lambda \right| \, dy.
    \end{multline}
    We divide the rest of the proof in two steps: the first treats the case $\operatorname{W} (0) \neq 0$ while the other addresses the case $\operatorname{W} (0) = 0$ and $\alpha \in \ell^{1, 2} (\Z)$.

    \medskip
    
    \noindent \emph{Step $1$.} Assume that $\operatorname{W} (0) \neq 0$. First, we prove that
    \begin{equation} \label{eqLowEn2}
       \sup_{x < 0 < y} \left| \int_{-\infty}^\infty \frac{\lambda e^{i t \lambda^2} \chi (\lambda^2)}{\operatorname{W} (\lambda)} \operatorname{f}_+ (\lambda, y) \operatorname{f}_- (\lambda, x) \, d \lambda \right|
       \lesssim |t|^{-\frac{1}{2}}.
    \end{equation}
    Let $\tilde{\chi}$ be a smooth, compactly supported cut-off such that $\tilde{\chi} (\lambda) = 1$ for any $\lambda \in \operatorname{supp} (\chi)$. Then,
    \begin{multline*}
        \sup_{x < 0 < y} \left| \int_{-\infty}^\infty \frac{\lambda e^{i t \lambda^2} \chi (\lambda^2)}{\operatorname{W} (\lambda)} \operatorname{f}_+ (\lambda, y) \operatorname{f}_- (\lambda, x) \, d \lambda \right| \\
        = \sup_{x < 0 < y} \left| \int_{-\infty}^\infty e^{i (t \lambda^2 + (x - y) \lambda )} \frac{\lambda \chi (\lambda^2)}{\tilde{\chi} (\lambda) \operatorname{W} (\lambda)} \operatorname{m}_+ (\lambda, y) \operatorname{m}_- (\lambda, x) \, d \lambda \right|.
    \end{multline*}
    We have that $\mathcal{F} (\tilde{\chi} \operatorname{W}) \in L^1 (\R)$ by Lemma \ref{lemFBound} applied with $j = 1$. Furthermore, $\operatorname{W} (\lambda) \neq 0$ for all $\lambda \in \operatorname{supp} (\tilde{\chi})$ by Proposition \ref{prpWronski} and by the assumption $\operatorname{W} (0) \neq 0$. Thus, by Lemma \ref{lemWiener}, we obtain that $\mathcal{F} (\tilde{\chi}^{-1} \operatorname{W}^{-1}) \in L^1 (\R)$. Furthermore, we have that $\mathcal{F} ( \lambda \chi (\lambda^2)) \in L^1 (\R)$, so that $\mathcal{F} ( \lambda \chi (\lambda^2) \tilde{\chi} (\lambda)^{-1} \operatorname{W} (\lambda)^{-1}) \in L^1 (\R)$. We obtain
    \begin{multline*}
        \left\| \mathcal{F} \left( \frac{\lambda \chi (\lambda^2) \operatorname{m}_+ (\lambda, y) \operatorname{m}_- (\lambda, x)}{\tilde{\chi} (\lambda) \operatorname{W} (\lambda)} \right) \right\|_{\mathcal{M}}
        = \left\| \mathcal{F} \left( \frac{\lambda \chi (\lambda^2)}{\tilde{\chi} (\lambda) \operatorname{W} (\lambda)} \right) \ast \mathcal{F} \left( \operatorname{m}_+ (\lambda, y) \right) \ast \mathcal{F} \left( \operatorname{m}_- (\lambda, x) \right) \right\|_{\mathcal{M}} \\
        \leq \left\| \mathcal{F} \left( \frac{\lambda \chi (\lambda^2)}{\tilde{\chi} (\lambda) \operatorname{W} (\lambda)} \right) \right\|_{L^1 (\R)} \left\| \mathcal{F} \left( \operatorname{m}_+ (\lambda, y) \right) \right\|_{\mathcal{M}} \left\| \mathcal{F} \left( \operatorname{m}_- (\lambda, x) \right) \right\|_{\mathcal{M}} \\
        \leq \left\| \mathcal{F} \left( \frac{\lambda \chi (\lambda^2)}{\tilde{\chi} (\lambda) \operatorname{W} (\lambda)} \right) \right\|_{L^1 (\R)} \left( \left\| \mathcal{F} \left( \operatorname{m}_+ (\lambda, y) - 1 \right) \right\|_{\mathcal{M}} + \left\| \delta_0 \right\|_{\mathcal{M}} \right) \\
        \times \left( \left\| \mathcal{F} \left( \operatorname{m}_- (\lambda, x) - 1 \right) \right\|_{\mathcal{M}} + \left\| \delta_0 \right\|_{\mathcal{M}} \right).
    \end{multline*}
    Using Lemma \ref{lemFourier}-\eqref{eqFourier1}, we obtain that
    \begin{equation*}
        \sup_{x < 0 < y} \left\| \mathcal{F} \left( \frac{\lambda \chi (\lambda^2) \operatorname{m}_+ (\lambda, y) \operatorname{m}_- (\lambda, x)}{\tilde{\chi} (\lambda) \operatorname{W} (\lambda)} \right) \right\|_{\mathcal{M}}
        < \infty.
    \end{equation*}
    Thus,
    \begin{equation*}
        \sup_{x < 0 < y} \left| \int_{-\infty}^\infty \frac{\lambda e^{i t \lambda^2} \chi (\lambda^2)}{\operatorname{W} (\lambda)} \operatorname{f}_+ (\lambda, y) \operatorname{f}_- (\lambda, x) \, d \lambda \right|
        \lesssim \left\| e^{i t \H_0} \psi \right\|_{L^\infty (\R)},
    \end{equation*}
    where $\psi = \mathcal{F} ( \lambda \chi (\lambda^2) (\tilde{\chi} (\lambda) \operatorname{W} (\lambda))^{-1} \operatorname{m}_+ (\lambda, y) \operatorname{m}_- (\lambda, x))$, leading the desired bound \eqref{eqLowEn2} by \eqref{eqFLDis}.

    Secondly, we prove that
    \begin{equation} \label{eqLowEn5}
       \sup_{0 \leq x < y} \left| \int_{-\infty}^\infty \frac{\lambda e^{i t \lambda^2} \chi (\lambda^2)}{\operatorname{W} (\lambda)}  \operatorname{f}_+ (\lambda, y) \operatorname{f}_- (\lambda, x)\, d \lambda \right|
       \lesssim |t|^{-\frac{1}{2}}.
    \end{equation}
    We recall (see \eqref{eqCoeff1}) that
    \begin{equation*}
        \operatorname{f}_- (\lambda, x)
        = \operatorname{a}_- (\lambda) \operatorname{f}_+ (\lambda, x) + \operatorname{b} (\lambda) \operatorname{f}_+ (-\lambda, x)
    \end{equation*}
    with
    \begin{equation*}
        \operatorname{a}_- (\lambda)
        = -\frac{W \left( \operatorname{f}_- (\lambda, \cdot), \operatorname{f}_+ (-\lambda, \cdot) \right)}{2 i \lambda}, \quad
        \operatorname{b} (\lambda)
        = -\frac{\operatorname{W} (\lambda)}{2 i \lambda}.
    \end{equation*}
    By Lemma \ref{lemFBound}, we obtain that $\mathcal{F} (\lambda \operatorname{a}_- (\lambda)) \in L^1 (\R)$. Furthermore, we have
    \begin{multline*}
        \left| \int_{-\infty}^\infty e^{i t \lambda^2} \frac{\lambda \chi (\lambda^2)}{\tilde{\chi} (\lambda) \operatorname{W} (\lambda)} \operatorname{f}_+ (\lambda, y) \operatorname{f}_- (\lambda, x) \, d \lambda \right|
        \lesssim \left| \int_{-\infty}^\infty e^{i t \lambda^2} e^{i \lambda (x + y)} \frac{\lambda \operatorname{a}_- (\lambda) \operatorname{m}_+ (\lambda, y) \operatorname{m}_+ (\lambda, x)}{\tilde{\chi} (\lambda) \operatorname{W} (\lambda)} \, d \lambda \right| \\
        + \left| \int_{-\infty}^\infty e^{i t \lambda^2} e^{i t (y - x)} \chi (\lambda) \operatorname{m}_+ (\lambda, y) \operatorname{m}_+ (-\lambda, x) \, d \lambda \right|.
    \end{multline*}
    Taking the supremum over $0 \leq x < y$, we prove \eqref{eqLowEn3} using Lemma \ref{lemWiener}, Lemma \ref{lemFourier} and \eqref{eqFLDis}.
    
    Finally, we obtain by symmetry that
    \begin{equation} \label{eqLowEn3}
       \sup_{x > y} \left| \int_{-\infty}^\infty \frac{\lambda e^{i t \lambda^2} \chi (\lambda^2)}{\operatorname{W} (\lambda)} \operatorname{f}_+ (\lambda, x) \operatorname{f}_- (\lambda, y) \, d \lambda \right|
       \lesssim |t|^{-\frac{1}{2}}.
    \end{equation}
    Injecting \eqref{eqLowEn2}, \eqref{eqLowEn5} and \eqref{eqLowEn3} in \eqref{eqLowEn4}, we obtain \eqref{eqLowEn1}.

    \medskip
    
    \noindent \emph{Step $2$.} Assume that $\alpha \in \ell^{1, 2} (\Z)$ and $\operatorname{W} (0) = 0$.
    Since $\tilde{\chi} (0) \operatorname{W} (0) = 0$ and $\mathcal{F} (\tilde{\chi} \operatorname{W}) \in L_1^1 (\R)$ by Lemma \ref{lemFBound}, we obtain that $\mathcal{F} (\tilde{\chi} (\lambda) \operatorname{W} (\lambda) \lambda^{-1})$ is well defined. Furthermore, as
    \begin{equation*}
        0
        = \tilde{\chi} (0) \operatorname{W} (0)
        = \int_{-\infty}^\infty \mathcal{F} (\tilde{\chi} \operatorname{W}) (\eta) \, d \eta,
    \end{equation*}
    we have
    \begin{align*}
        \mathcal{F} \left( \frac{\tilde{\chi} (\lambda) \operatorname{W} (\lambda)}{\lambda} \right) (\xi)
        & = \pi i \left( \int_{\xi}^\infty \mathcal{F} (\tilde{\chi} \operatorname{W}) (\xi) \, d \eta - \int_{-\infty}^\xi \mathcal{F} (\tilde{\chi} \operatorname{W}) (\xi) \, d \eta \right) \\
        & = 2 \pi i \int_\xi^\infty \mathcal{F} (\tilde{\chi} \operatorname{W}) (\xi) \, d \eta
        = -2 \pi i \int_{-\infty}^\xi \mathcal{F} (\tilde{\chi} \operatorname{W}) (\xi) \, d \eta
    \end{align*}
    Thus,
    \begin{align*}
        \int_{0}^\infty \left| \mathcal{F} \left( \frac{\tilde{\chi} (\lambda) \operatorname{W} (\lambda)}{\lambda} \right) (\xi) \right| \, d \xi
        & \lesssim \int_{0}^\infty \int_\xi^\infty \left| \mathcal{F} (\tilde{\chi} \operatorname{W}) (\eta) \right| d \eta \, d \xi \\
        & \lesssim \int_{-\infty}^\infty \left| \eta \mathcal{F} (\tilde{\chi} \operatorname{W}) (\eta) \right| d \eta
        < \infty,
    \end{align*}
    where we used Fubini's theorem for the third inequality. Similarly, we obtain
    \begin{equation*}
        \int_{0}^\infty \left| \mathcal{F} \left( \frac{\tilde{\chi} (\lambda) \operatorname{W} (\lambda)}{\lambda} \right) (\xi) \right| \, d \xi
        < \infty.
    \end{equation*}
    Thus, $\mathcal{F} (\tilde{\chi} (\lambda) \operatorname{W} (\lambda) \lambda^{-1}) \in L^1 (\R)$, and so is $\mathcal{F} (\tilde{\chi} (\lambda) \operatorname{b} (\lambda))$. Thus, Lemma \ref{lemWiener} together with Proposition \ref{prpNonZero} and the assumption $\alpha \in \ell^{1, 2} (\Z)$ ensure that $\mathcal{F} \left( (\tilde{\chi} (\lambda) \operatorname{b} (\lambda))^{-1} \right) \in L^1 (\R)$.
    Furthermore, using Lemma \ref{lemFBound} for $j = 2$, we obtain that $\mathcal{F} (\operatorname{a}_-) \in L^1 (\R)$. Re-writing
    \begin{equation*}
        \int_{-\infty}^\infty \frac{\lambda e^{i t \lambda^2} \chi (\lambda^2)}{\tilde{\chi} (\lambda) \operatorname{W} (\lambda)} \operatorname{f}_+ (\lambda, y) \operatorname{f}_- (\lambda, x) \, d \lambda
        = -\frac{1}{2 i} \int_{-\infty}^\infty \frac{e^{i t \lambda^2} \chi (\lambda^2)}{\tilde{\chi} (\lambda) \operatorname{b} (\lambda)} \operatorname{f}_+ (\lambda, y) \operatorname{f}_- (\lambda, x) \, d \lambda,
    \end{equation*}
    we can use similar arguments than in Step $1$ of the proof to obtain \eqref{eqLowEn1}. 
    This concludes the proof.
\end{proof}

We can finally prove Theorem \ref{thMain}.

\begin{proof}[Proof of Theorem \ref{thMain}]
    The equation \eqref{eqMain} is a concatenation of the high-energy estimate stated in Proposition \ref{prpHighEn}-\eqref{eqHighEn1} and the low-energy estimate stated in Proposition \ref{prpLowEn}-\eqref{eqLowEn1}. This concludes the proof.
\end{proof}

\bibliographystyle{alpha}
\bibliography{bibliography}

\begin{thebibliography}{AGHKH12}

\bibitem[AGHKH12]{AlGeHoHo12}
Sergio Albeverio, Friedrich Gesztesy, Raphael Hoegh-Krohn, and Helge Holden.
\newblock {\em Solvable models in quantum mechanics}.
\newblock Springer Science \& Business Media, 2012.

\bibitem[Agm75]{Ag75}
Shmuel Agmon.
\newblock {Spectral properties of {S}chr{\"o}dinger operators and scattering
  theory}.
\newblock {\em Annali della Scuola Normale Superiore di Pisa-Classe di
  Scienze}, 2(2):151--218, 1975.

\bibitem[APF14]{FePa14}
Jaime Angulo~Pava and Lucas C.~F. Ferreira.
\newblock On the {Schr{\"o}dinger} equation with singular potentials.
\newblock {\em Differ. Integral Equ.}, 27(7-8):767--800, 2014.

\bibitem[AS05]{AdSa05}
Riccardo Adami and Andrea Sacchetti.
\newblock The transition from diffusion to blow-up for a nonlinear
  {S}chr{\"o}dinger equation in dimension 1.
\newblock {\em Journal of Physics A: Mathematical and General},
  38(39):8379--8392, 2005.

\bibitem[BI14]{BaIg14}
Valeria Banica and Liviu Ignat.
\newblock {Dispersion for the {S}chr{\"o}dinger equation on the line with
  multiple Dirac delta potentials and on delta trees}.
\newblock {\em Analysis \& PDE}, 7(4):903--927, 2014.

\bibitem[Caz03]{Ca03}
Thierry Cazenave.
\newblock {\em {Semilinear {S}chr\"odinger equations}}, volume~10 of {\em
  Courant Lecture Notes in Mathematics}.
\newblock New York University / Courant Institute of Mathematical Sciences, New
  York, 2003.

\bibitem[CMY19]{CoMiYa19}
Horia~D Cornean, Alessandro Michelangeli, and Kenji Yajima.
\newblock Two-dimensional {S}chr{\"o}dinger operators with point interactions:
  threshold expansions, zero modes and ${L}^p$-boundedness of wave operators.
\newblock {\em Reviews in Mathematical Physics}, 31(04):1950012, 2019.

\bibitem[DH09]{DaHo09}
Kiril Datchev and Justin Holmer.
\newblock Fast soliton scattering by attractive delta impurities.
\newblock {\em Communications in Partial Differential Equations},
  34(9):1074--1113, 2009.

\bibitem[DMSY18]{DeMiScYa18}
Gianfausto Dell’Antonio, Alessandro Michelangeli, Raffaele Scandone, and
  Kenji Yajima.
\newblock ${L}^p$-boundedness of wave operators for the three-dimensional
  multi-centre point interaction.
\newblock In {\em Annales Henri Poincar{\'e}}, volume~19, pages 283--322.
  Springer, 2018.

\bibitem[DMW11]{DuMaWe11}
Vincent Duch{\^e}ne, Jeremy~L Marzuola, and Michael~I Weinstein.
\newblock Wave operator bounds for one-dimensional {S}chr{\"o}dinger operators
  with singular potentials and applications.
\newblock {\em Journal of Mathematical Physics}, 52(1), 2011.

\bibitem[dPT06]{AnPiTe06}
Piero d'Ancona, Vittoria Pierfelice, and Alessandro Teta.
\newblock Dispersive estimate for the {S}chr{\"o}dinger equation with point
  interactions.
\newblock {\em Mathematical methods in the applied sciences}, 29(3):309--323,
  2006.

\bibitem[DT79]{DeTr79}
Percy Deift and Eugene Trubowitz.
\newblock {Inverse scattering on the line}.
\newblock {\em Communications on Pure and Applied Mathematics}, 32:121--251,
  1979.

\bibitem[Dur70]{Du70}
Peter~L. Duren.
\newblock {\em Theory of HP spaces}.
\newblock Pure and applied mathematics ; 38. 810852004. Academic Press, New
  York, 1970.

\bibitem[Fad64]{Fa64}
Lioudvig~D. Faddeev.
\newblock {Svojstva $S$-matricy odnomernogo uravneniya Shredingera}.
\newblock In {\em Kraevye zadachi matematicheskoj fiziki. Vol.~2}, volume~73 of
  {\em Trudy Matematicheskogo Instituta im. V.~A.~Steklova}, pages 314--336.
  Nauka, Moscow--Leningrad, 1964.
\newblock English translation: \emph{Properties of the $S$-matrix of the
  one-dimensional {S}chr\"odinger equation}, Amer. Math. Soc. Transl. (2)
  \textbf{65} (1967), 139--166.

\bibitem[GS04]{GoSc04}
Michael Goldberg and Wilhelm Schlag.
\newblock {Dispersive estimates for {S}chr{\"o}dinger operators in dimensions
  one and three}.
\newblock {\em Communications in mathematical physics}, 251(1):157--178, 2004.

\bibitem[GV79a]{GiVe79-b}
Jean Ginibre and G~Velo.
\newblock On a class of nonlinear {S}chr{\"o}dinger equations. {II}.
  {S}cattering theory, general case.
\newblock {\em Journal of Functional Analysis}, 32(1):33--71, 1979.

\bibitem[GV79b]{GiVe79-a}
Jean Ginibre and Giorgio Velo.
\newblock On a class of nonlinear {S}chr{\"o}dinger equations. {I}. {T}he
  {C}auchy problem, general case.
\newblock {\em Journal of Functional Analysis}, 32(1):1--32, 1979.

\bibitem[IS17]{IaSc17}
Felice Iandoli and Raffaele Scandone.
\newblock Dispersive estimates for {S}chr{\"o}dinger operators with point
  interactions in r3.
\newblock In {\em Advances in Quantum Mechanics: Contemporary Trends and Open
  Problems}, pages 187--199. Springer, 2017.

\bibitem[Kat76]{Ka66}
Tosio Kato.
\newblock {\em {{Perturbation theory for linear operators}}}.
\newblock Grundlehren der mathematischen Wissenschaften : a series of
  comprehensive studies in mathematics. Springer, Berlin, 1976.

\bibitem[Kat04]{Ka04}
Yitzhak Katznelson.
\newblock {\em {An Introduction to Harmonic Analysis}}.
\newblock Cambridge Mathematical Library. Cambridge University Press, 3
  edition, 2004.

\bibitem[KM13]{KoMa13}
Oleksiy Kostenko and Mark Malamud.
\newblock 1-{D} {S}chrodinger operators with local point interactions: a
  review.
\newblock {\em Spectral Analysis, Integrable Systems, and Ordinary Differential
  Equations}, pages 235--262, 2013.

\bibitem[KP31]{KrPe31}
Ralph de~Laer Kronig and William~George Penney.
\newblock Quantum mechanics of electrons in crystal lattices.
\newblock {\em Proceedings of the royal society of London. series A, containing
  papers of a mathematical and physical character}, 130(814):499--513, 1931.

\bibitem[KS10]{KoSa10}
Hynek Kova{\v{r}}{\'\i}k and Andrea Sacchetti.
\newblock A nonlinear {S}chr{\"o}dinger equation with two symmetric point
  interactions in one dimension.
\newblock {\em Journal of Physics A: Mathematical and Theoretical},
  43(15):155205, 2010.

\bibitem[Mar18]{Ma18}
Alexandre Martin.
\newblock {On the limiting absorption principle for a new class of
  {S}chr{\"o}dinger Hamiltonians}.
\newblock {\em Confluentes Mathematici}, 10(1):63--94, 2018.

\bibitem[RS81]{ReSi81}
Michael Reed and Barry Simon.
\newblock {\em {I}: {F}unctional analysis}, volume~1.
\newblock Academic press, 1981.

\bibitem[Sch07]{Sc07}
Wilhelm Schlag.
\newblock Dispersive estimates for {S}chr{\"o}dinger operators: a survey.
\newblock {\em Mathematical aspects of nonlinear dispersive equations},
  163:255--285, 2007.

\bibitem[SS07]{Su07}
Catherine Sulem and Pierre-Louis Sulem.
\newblock {\em The nonlinear {S}chr{\"o}dinger equation: self-focusing and wave
  collapse}, volume 139.
\newblock Springer Science \& Business Media, 2007.

\end{thebibliography}

\end{document}